\documentclass[12pt]{amsart}

\usepackage[numbers,sort&compress]{natbib}
\usepackage{amsfonts,bbm}
\usepackage{mathrsfs} 
\usepackage{appendix}
\usepackage{mathtools}
\usepackage{amsmath,amssymb,amsthm,bbm}
\usepackage{latexsym}
\RequirePackage{natbib}
\usepackage{tikz}
\usetikzlibrary{arrows,positioning,chains,fit,shapes,calc,decorations}

\newtheorem{theorem}{Theorem}[section]
\newtheorem{lemma}[theorem]{Lemma}
\newtheorem{corollary}[theorem]{Corollary}
\newtheorem{proposition}[theorem]{Proposition}

\newtheorem{example}[theorem]{Example}
\newtheorem{remark}[theorem]{Remark}

\usepackage{hyperref,url}
\usepackage[capitalise]{cleveref}

\def\RR{{\mathbb R}}  
\def\PP{{\mathbb P}}
\def\NN{{\mathbb N}} 
\def\EE{{\mathbb E}}        
\def\FF{{\mathcal F}}      
\def\SS{{\mathcal S}}      
\def\AA{{\mathcal A}}      
\def\Var{\theta}     
\def\normeq{\mu}      
\def\sumprob{\delta}
      
\newcommand{\norm}[1]{\left\Vert #1 \right\Vert} 
\newcommand{\escon}[2]{\mathbb E (#1 \, \vert \,  #2)}

\DeclareMathOperator{\fix}{Fix}

\def\Var{\theta}     
\def\normeq{\mu}      
\def\sumprob{\nu}
\def\optv{{\overline{r}}}

\usepackage[textwidth=30mm]{todonotes}

\begin{document}

\title[Stochastic fixed-point iterations for nonexpansive maps]{Stochastic  fixed-point iterations for nonexpansive maps: Convergence and error bounds}

\author[M. Bravo]{Mario Bravo}
\address[M.B.]{Universidad Adolfo Ib\'a\~nez, Facultad de Ingenier\'ia y Ciencias,
Diagonal Las Torres 2640, Santiago, Chile} 
\email{ \href{mailto:mario.bravo@uai.cl}{\nolinkurl{mario.bravo@uai.cl}}}
\thanks{{\em Acknowledgements.} The work of Mario Bravo was partially funded by FONDECYT grant No. 1191924 and Anillo Grant ANID/ACT210005. Roberto Cominetti acknowledges the support of FONDECYT No.1171501. } 

\author[R. Cominetti]{Roberto Cominetti}
\address[R.C.]{Universidad Adolfo Ib\'a\~nez, Facultad de Ingenier\'ia y Ciencias,
Diagonal Las Torres 2640, Santiago, Chile} 
\email{ \href{mailto:roberto.cominetti@uai.cl}{\nolinkurl{roberto.cominetti@uai.cl}}}

\begin{abstract}
We study a stochastically perturbed version of the well-known Krasnoselski--Mann iteration for  computing fixed points of nonexpansive 
maps in finite dimensional normed spaces.
We discuss sufficient conditions on the stochastic noise and stepsizes that guarantee almost sure convergence of the iterates towards 
a fixed point, and derive non-asymptotic error bounds and convergence rates for the fixed-point residuals. 
Our main results concern the case of a martingale difference noise with variances that can possibly grow unbounded.
This supports an application to reinforcement learning for average reward Markov decision processes, for which we establish
convergence and asymptotic rates. We also 
analyze in depth the case where the noise has uniformly bounded variance, obtaining error bounds 
with explicit computable constants. 
 \end{abstract}
 

\maketitle

\section{Introduction}
Consider a map $T:\RR^d\to \RR^d$ which is nonexpansive in some given  norm $\|\cdot\|$ in $\RR^d$  (not necessarily Euclidean),
and assume that its set of fixed points $\fix(T)$ is nonempty. 
Starting from an arbitrary initial point $x_0\in \RR^d$, we investigate the convergence  of the Stochastic Krasnoselskii-Mann iteration
\begin{equation}\label{Eq:skm} 
(\forall\,n\geq 1)\quad  x_{n}= (1\!-\!\alpha_{n})\,x_{n-1} + \alpha_{n}\,(Tx_{n-1} +U_{n}) \tag{\textsc {skm}}
\end{equation}
 where $(\alpha_n)_{n\geq 1}$ is a sequence of averaging  scalars, and $(U_n)_{n\geq 1}$ is a
random process in $\RR^d$ adapted to some filtration $(\mathcal F_n)_{n\geq 1}$  on a probability space 
$(\Omega, \mathcal F, \PP)$.  

The iteration \eqref{Eq:skm} is a classical Robbins--Monro process  in which  
the evaluation of $Tx_{n-1}$ is subject to an additive random noise $U_n$. It appears naturally when 
considering gradient-type and decomposition methods in stochastic convex optimization, and 
in machine learning applications where the objective function and its gradient are estimated
by sampling from data. 

The stochastic approximation (SA) approach to study these type of iterations has a long history  
started with the landmark papers by  Robbins--Monro \cite{rbm} and Kiefer--Wolfowitz \cite{kw}
in the context of statistical estimation and regression.
For an historical perspective of the use of SA in stochastic optimization we refer to the influential paper by Nemirovskii {\em et al.} 
\cite{nem}. Closely related to stochastic optimization algorithms are operator splitting methods for maximal monotone 
inclusions, which can also be framed in forms similar to \eqref{Eq:skm} ({\em cf.} \cite{cp1,cp3}).
While these works are mostly concerned with nonexpansive maps in Euclidean settings, 
\eqref{Eq:skm} is also a template that fits other discrete-time stochastic dynamics 
such as $Q$-learning iterations for average reward 
Markov decision processes \cite{abb1,jjs,tsi},  in which the relevant operator is nonexpansive in the infinity norm $\|\cdot\|_\infty$.

\subsection{Our contribution}
In this paper we
focus on the general framework described by \eqref{Eq:skm}. 
We begin by establishing general conditions under which the sequences $(x_n)_{n\in\NN}$ generated by \eqref{Eq:skm}
are guaranteed to converge almost surely  towards some random fixed point $x^*\in\fix(T)$. 
Here and after, all equations and statements involving random elements must be understood
to hold  
$\PP$-almost surely.  We then proceed to derive 
explicit error bounds for the residuals $\|x_n-Tx_n\|$, both in expectation and with high probability, 
with a special focus in the case where $\alpha_n=1/(n+1)^a$ with $\frac{1}{2}\leq a\leq 1$. 
These seem to be the first rates for the general
 \eqref{Eq:skm} iteration.

Our main results concern the case  where $U_n$ is a martingale difference sequence in $L^2$, and
under standard assumptions on the stepsizes $\alpha_n$, namely
\begin{flalign}\label{H0}\tag{$\textsc{h}_0$}
(\forall n\geq 1)&\hspace{4ex}\mbox{$  \alpha_n\!\in\!(0,1)$\hspace{1ex} and\hspace{1ex}  $\sum_{k=1}^\infty\alpha_k(1\!-\!\alpha_k) =\infty$,}\\
\label{H1} \tag{\textsc{h}\textsubscript{1}} 
 (\forall n\geq 1)&\hspace{4ex}\mbox{$\escon{U_{n}}{\FF_{n-1}}=0$\hspace{1ex} and\hspace{1ex} $\Var_n^2\triangleq \EE  ( \norm{U_n}^2_2 ) <\infty$,} 
 \end{flalign}
where $\norm{\,\cdot\,}_2$ denotes the standard $L_2$-norm in $\RR^d$. 
Although we  will study in detail the special case of uniformly bounded variances,
our general results allow $\Var_n^2$  to diverge in a controlled manner. This will prove very useful for
the application to the RVI-$Q$-learning iteration presented in Section \S\ref{sec:RVI-Q}.  We also consider the case where $U_n$  includes additional stochastic terms, not necessarily
of martingale type nor adapted, but which vanish asymptotically as $n$ grows. The following is a brief overview of our results. 

\vspace{2ex}

\noindent{\bf\em \ \hspace*{0.3cm}Convergence and error bounds:} 
\vspace{1ex}

\begin{itemize}\setlength\itemsep{0.5em}
    \item Under mild conditions on $\alpha_n$  and $\Var_n$, \cref{Thm:First} shows that the sequence $(x_n)_{n\in\NN}$ generated 
    by \eqref{Eq:skm} converges almost surely to some point in $\fix(T)$. In particular, this applies when the $\Var_n$'s are bounded and $\alpha_n= {1}/{(n+1)^a}$ with $\frac{2}{3}<a\leq 1$.  
    \item  \cref{Thm:Vanishing_Noise} complements this result by showing that $(x_n)_{n\in\NN}$ also converges to some point in $\fix(T)$ when the noise $U_n$ converges to zero sufficiently fast.    
    \item \cref{Thm:Rate_Expectation,Thm:convolution_bound} then establish error bounds for the expected fixed-point residuals $\EE(\norm{x_n-Tx_n})$, and \cref{Cor:error_bound_high_prob} derives error bounds in  high probability.
\end{itemize}

\vspace{1.5ex}
\noindent{\bf\em \ \hspace*{0.3cm}Special cases:}
\vspace{1ex}

\begin{itemize}\setlength\itemsep{0.5em}
    
    \item Section \S\ref{sec:RVI-Q} presents an application to average reward Markov decision processes, establishing
    convergence and error bounds for the {\em RVI-$Q$-learning algorithm}.

    \item Section \S\ref{sec:BndVar} studies in depth the case of uniformly bounded variances $\Var_n^2\leq\Var^2$.
    For constant stepizes $\alpha_n\equiv\alpha$ over a fixed horizon $n_0$, the best bound we get is $\EE(\|x_{n_0}-Tx_{n_0}\|)\leq C/\sqrt[6]{n_0}$ with an explicit  constant $C$.
     Similarly,  for  stepsizes $\alpha_n= 1/(n+1)^a$ with $\frac{1}{2} \leq a\leq 1$ we derive explicit error bounds whose best rate
     is attained when $a=\frac{2}{3}$ with  $\EE(\|x_n-Tx_n\|)\leq C\ln(n)/\sqrt[6]{n}$. We illustrate these results by an application to stochastic gradient descent.
    \item Section \S\ref{Sec:hilbert} presents an alternative direct analysis of error bounds in the case of Euclidean spaces, and compares them with our results for arbitrary norms.
    \end{itemize}

 \subsection{Some notation}
We  distinguish the original norm $\|\cdot\|$ for which the map $T$ is nonexpansive, from the standard Euclidean norm in $\RR^d$ denoted by $\|\cdot\|_2$.
Since all norms in $\RR^d$ are equivalent, throughout we denote $\mu>0$  a constant such that
\begin{equation} \label{H2p}
(\forall x\!\in\!\RR^d)\qquad \|x\|\leq\mu\|x\|_2. 
\end{equation}

For any two positive sequences we write $a_n\sim b_n$ when their quotient $a_n/b_n$ is bounded from above and bounded away from 0.
We also write $a_n\approx b_n$ when $\lim_{n\to\infty} a_n/b_n=1$.

\subsection{Related work}\label{sec:related}

A classical iteration of the form \eqref{Eq:skm} is the stochastic gradient descent for minimizing an expected cost $f(x)=\EE[F(x,\xi)]$, namely
\begin{equation}\label{Eq:sgd}
\tag{\textsc{sgd}}
x_{n}=x_{n-1}-\gamma_n\nabla_x F(x_{n-1},\xi_n)
\end{equation}
 where the $\xi_n$'s are i.i.d. samples of the random variable $\xi$. In a Euclidean setting and assuming that $f(\cdot)$ is convex with an $L$-Lipschitz gradient, the iteration \eqref{Eq:sgd} coincides with \eqref{Eq:skm} for the nonexpansive map $Tx=x-\frac{2}{L}\nabla f(x)$, with $\alpha_n=\frac{\gamma_n L}{2}$
 and $U_n=\frac{2}{L}[\nabla f(x_n)-\nabla_xF(x_n,\xi_n)]$. Under suitable conditions discussed later in \S\ref{subsec:SGD}, 
Bach and Moulines
\cite{bm} gave explicit error bounds for the expected optimality 
gap $\EE\big(f(x_n)-\min f\big)$, achieving a rate $O(1/\sqrt[3]{n})$ with stepsizes $\gamma_n\sim 1/n^{2/3}$,
and faster rates when $f(\cdot)$ is strongly convex. Error bounds with faster rates were also
obtained earlier in \cite[Polyak and Judistky]{py} and  \cite[Nemirovskii {\em et al.}]{nem},
for {\em gradient-projection} methods  in convex stochastic optimization over a compact set.
Our results apply to general nonexpansive maps and do not assume the existence of a potential function, so  
that the optimality gap is no longer meaningful. We consider instead the expected fixed-point residual, 
which for \eqref{Eq:sgd} corresponds to $\EE\big(\|x_n-Tx_n\|\big)=\frac{2}{L}\EE\big(\|\nabla f(x_n)\|\big)$.
As observed in \cite{AZ18}, there are situations in which obtaining a small gradient is
more relevant than attaining a small optimality gap.
In \S\ref{subsec:SGD} we apply our general results to obtain error bounds for the expected gradient norm,
and compare with those derived from \cite{bm} as well as with more recent work.

Beyond the setting of stochastic optimization, Combettes and Pesquet \cite{cp1,cp3} studied
a general class of stochastic iterations for monotone inclusions that include  \cref{Eq:skm} as a special case.
Their analysis is restricted to Hilbert spaces and is based on the notion of quasi-Fej\'er sequences. The noise $U_n$ is required to vanish sufficiently fast as $n\to\infty$, which unfortunately rules out  the 
simplest but relevant case of an i.i.d. noise sequence. They establish the almost sure convergence of the iterates, but no error bounds nor convergence rates are presented.
On the positive side, their results apply to more general
stochastic versions of the Douglas-Rachford and forward-backward splitting methods for structured maximal monotone inclusions,
in both synchronous and asynchronous implementations.
Concurrently with \cite{cp1,cp3}, Rosasco, Villa and Vu  studied in \cite{rvv} a  stochastic forward-backward iteration for 
monotone inclusions involving a sum of a maximal monotone set-valued map and a co-coercive operator. Their results support more general noise sequences
which can be state-dependent with a  variance that may grow quadratically with the state.
Under suitable conditions on the stepsizes and a uniform monotonicity condition, they prove the almost sure convergence of the iterates
towards the unique solution $x^*$ of the problem. Also, under strong monotonicity, they provide non-asymptotic error bounds for the expected squared 
distance $\EE(\|x_n-x^*\|^2)$.
Similar iterations, with state-dependent noise whose variance grows in a controlled manner, were 
studied for  stochastic variational inequalities \cite[Iusem {\em et al.}]{ijot} and stochastic optimization \cite[Jofr\'e and Thomson]{jt}, attaining near-optimal oracle complexity
 by using a variance reduction strategy based on mini-batches whose size grows superlinearly.

A classical approach to analyze iterations of the form \eqref{Eq:skm} is the well-known ODE method for stochastic approximation algorithms (see \cite{Benaim99,ky} for a complete account and references). The main idea is to connect the asymptotic behavior of \eqref{Eq:skm} to the one of the continous dynamics $\dot x= Tx-x$. Under suitable assumptions, the limit set of the process $(x_n)_{n\in\NN}$ is contained almost surely in an {\em Internally Chain Transitive} (ICT) set of this ODE, that is to say, a compact invariant set with no proper attractors. In the simplest case, when the ODE has a unique global attractor it is the only ICT set. 
However, in general ICT's can have a complex structure and provide little information about the limit of the process. Convergence towards a single point can still be obtained  
under further structure of the underlying ODE, such as when the dynamics are cooperative \cite[Bena\"im and Faure]{BF12}. Up to our knowledge, in our setting of a general nonexpansive map $T$, it is not even known whether $\fix(T)$ is an ICT set, nor whether this technique can yield convergence to a single point.
The ODE approach was used by Abounadi et al. \cite{abb2} to study the convergence of \eqref{Eq:skm}, as well as an asynchronous version, although under more restrictive assumptions:  the noise sequence $U_n$ is required to be uniformly bounded, and  $\fix(T)=\{x^*\}$ is assumed to be a singleton and a global attractor for the dynamics $\dot x= Tx-x$
(see \cite[Assumption 3.2]{abb2}).
Moreover, \cite{abb2} focuses mainly on the convergence towards $x^*$ and does not provide error bounds nor convergence rates.

Our paper departs from these previous trends. Instead of using the ODE approach for stochastic approximation, or restricting to Hilbert spaces and using quasi-Fej\'er convergence, we base our analysis on an explicit error bound for inexact Krasnoselskii-Mann iterations established in 
\cite[Bravo, Cominetti, and Pavez]{bcp}. This technique allows to prove the convergence of the \eqref{Eq:skm} iterates in general normed spaces, and provides  error bounds and convergence rates
for the fixed-point residual of the last-iterate.

\subsection{Structure of the paper}
Section \S{}\ref{Section:2} presents our general results on the almost sure convergence of the iteration
 \eqref{Eq:skm} under martingale difference noise and vanishing stochastic perturbations, 
followed by the analysis of error bounds and convergence rates for the expected fixed-point residuals.
 Section \S{}\ref{sec:RVI-Q} applies the general results in the context of $Q$-Learning 
 for average reward Markov decision processes. Section  \S{}\ref{sec:BndVar} studies the case when the 
 noise has uniformly bounded variance, presenting more explicit 
 error bounds for the most common cases of constant stepsizes $\alpha_n\equiv\alpha$ and stepsizes of power form $\alpha_n={1}/{(n+1)^a}$,
which are then  illustrated through an application in stochastic optimization. 
 Finally, Section \S\ref{Sec:hilbert} presents a short discussion of the Euclidean case, putting in perspective  our results for general normed spaces. 
  \cref{appendix00,appendix1} collect the more technical estimates used in the proofs.

\newpage
\section{Almost sure convergence and error bounds}\label{Section:2}
Our main tool for studying the convergence of \eqref{Eq:skm}
 is the estimate \eqref{Eq:IKMbound} below (taken from \cite{bcp}) for inexact Krasnoselskii--Mann iterations of the form
\begin{equation}\label{Eq:ikm}\tag{\textsc{ikm}}
\left\{\begin{array}{l}
z_0=x_0\\
z_{n}=(1\!-\!\alpha_{n})\,z_{n-1} + \alpha_{n}\,(Tz_{n-1}+e_{n})\quad\forall n\geq 1
\end{array}\right.
\end{equation}
where $e_{n}\in\RR^d$ is an additive perturbation of $Tz_{n-1}$.
For the rest of this paper we introduce the sequence $\tau_n$ and the real function  $\sigma:(0,\infty)\to(0,\infty)$ defined as
\begin{equation*}
\left\{\begin{array}{l}
\tau_n=\mbox{$\sum_{k=1}^n\alpha_k(1\!-\!\alpha_k)$},\\[1ex]
\sigma(y)=\min\{1,1/\sqrt{\pi y}\}.
\end{array}\right.
\end{equation*}

\begin{theorem}\label{Thm:BCP} If $\kappa\geq 0$ is such that  $\|Tz_n-x_0\|\leq\kappa$ for all $n\geq 1$, then 
\begin{equation}\label{Eq:IKMbound}
\|z_n-Tz_n\|\leq \kappa\,\sigma(\tau_n)+\sum_{k=1}^n2\,\alpha_k\|e_k\|\,\sigma(\tau_n\!-\!\tau_k)+2\,\|e_{n+1}\|.
\end{equation}
Moreover, if $\tau_n\to\infty$ and $\|e_n\|\to 0$ with $S\triangleq\sum_{n=1}^\infty\alpha_n\|e_{n}\|<\infty$, then
\eqref{Eq:IKMbound} holds with $\kappa=2\mathop{\rm dist}(x_0,\fix(T))+S$, and we have $\|z_n-Tz_n\|\to 0$ as well as
$z_n\to x^*$ for some fixed point $x^*\in\fix(T)$.
\end{theorem}
\begin{proof} The bound \eqref{Eq:IKMbound} was established in \cite[Theorem 1]{bcp}. 
Also, when $\tau_n\to\infty$, $\|e_n\|\to 0$, and $S<\infty$,  from \cite[Proposition 2]{bcp} we have  
that $\|Tz_n-x_0\|\leq\kappa$ holds for  $\kappa=2\mathop{\rm dist}(x_0, \fix(T)) + S$,
and using \cite[Theorem 1]{bcp} again we get $\|z_n-Tz_n\|\to 0$.
The convergence of $z_n$ follows by a standard argument. Namely, consider an arbitrary fixed point $\bar x\in\fix(T)$.
The triangle inequality and non-expansivity of $T$ yields
\begin{eqnarray*}
\norm{z_{n} - \bar x}&=&\norm{(1\!-\!\alpha_{n})(z_{n-1} - \bar x) + \alpha_{n}\,(Tz_{n-1} - T\bar x+e_{n})}\\
&\leq &\norm{z_{n-1} - \bar x}  + \alpha_{n}\|e_{n}\|.
\end{eqnarray*}
Adding $\varepsilon_n= \sum_{k> n}\alpha_{k}\|e_{k}\|$ on both sides 
it follows that $\|z_n-\bar x\|+\varepsilon_{n}$  is decreasing and hence convergent.
Since $\varepsilon_n\to 0$ we conclude that $\norm{z_{n} - \bar x}$ converges as well. 
Hence $z_n$ is bounded and  has a cluster point $x^*$. Now,  $\|z_n-Tz_n\|\to 0$ implies that $x^*\in\fix(T)$
 and therefore $\|z_n-x^*\|$ converges, so that in fact $\|z_n-x^*\|\to 0$.
\end{proof}
\begin{remark}
The bound \eqref{Eq:IKMbound} is also trivially guaranteed when the operator
 $T$ has a bounded range:  just take any  $\kappa \geq \sup_{x\in \RR^d} \|Tx- x_0\|$.
\end{remark}

\vspace{1ex}
\subsection{A simple reduction trick} A direct application of \cref{Thm:BCP} to \eqref{Eq:skm} would require $\|U_n\|\to 0$ and
$\sum_{n\geq 1}\alpha_n\|U_n\|<\infty$, which is rather restrictive as it rules out even the 
simplest case where the $U_n$'s are i.i.d.. Fortunately, a reformulation of \eqref{Eq:skm}
using an auxiliary averaged sequence $\overline{U}_n$ yields convergence 
under much weaker conditions. Namely, let $\overline{U}_0=0$ and 
defined recursively $(\overline U_n)_{n\in\NN}$   by 
\begin{equation}\label{Eq:AveragedNoise} \tag{\textsc u}
  \overline U_{n}= (1\!-\!\alpha_{n})\,\overline U_{n-1} + \alpha_{n}\,U_{n}.
\end{equation}
Substracting this from \eqref{Eq:skm} and denoting $z_n = x_n - \overline U_n$ we get
$$z_{n}=(1\!-\!\alpha_{n})\,z_{n-1} + \alpha_{n}\,Tx_{n-1}$$
which is precisely of the form \eqref{Eq:ikm} with $e_{n}=Tx_{n-1}-Tz_{n-1}$ satisfying
\begin{equation*}
\|e_{n}\|\leq \|x_{n-1}-z_{n-1}\|=\|\overline{U}_{n-1}\|.
\end{equation*}

The advantage of this alternative form of the iteration is that, in contrast with the original noise $U_n$, 
the averages $\overline{U}_n$ may converge to zero fast enough so that  the conditions of  \cref{Thm:BCP}
are met with high probability or even almost surely. As an illustration, in the i.i.d. case with $\alpha_n={1}/{n}$ we have
$\overline{U}_n=\frac{1}{n}\sum_{k=1}^nU_k$ and the Central Limit Theorem gives
$\|\overline{U}_n\|\sim O({1}/{\sqrt{n}})$ almost surely. 
In what follows we identify conditions on the sequence $(\alpha_n)_{n\in\NN}$ which 
guarantee that $\|\overline{U}_n\|\to 0$ and $\sum_{n\geq 1}\alpha_n\|\overline{U}_{n-1}\|<\infty$, 
implying the convergence of the iterates $z_n$, and {\em a fortiori} of $x_n$, towards some fixed point $x^*\in\fix(T)$. 
We stress that, although our arguments exploit the averaged noise process $\overline{U}_n$, we 
derive the convergence of the original iterates $x_n$, and not just 
ergodic convergence of some averaged iterates. 

\begin{remark} The previous rewriting of \eqref{Eq:skm} in terms of the sequences 
$(z_n)_{n\in\NN}$ and $(\overline{U}_n)_{n\in\NN}$, reveals that  \eqref{Eq:skm}
 carries implicitly a form of variance reduction that operates automatically as the iterates evolve. We will
come back to this issue in \cref{remark3.6}, after establishing our error bounds for the fixed-point residuals.
\end{remark}

\subsection{Convergence of iterates for martingale difference noise}
We proceed to state our first convergence result, which concerns the case where the noise $U_n$ is a martingale difference sequence satisfying \eqref{H1}. For the rest of this paper, for each $n\in\NN$ and $0\leq k\leq n$ we define the quantities
\begin{eqnarray*}
 \pi_k^n&=&\mbox{$\alpha_k\prod_{i=k+1}^n(1\!-\!\alpha_i)$},\\[1ex]
\sumprob_n&=&\mbox{$\sqrt{\sum_{k=1}^n(\pi_k^n)^2\Var_k^2}$},
\end{eqnarray*}
and establish the following straightforward estimate with $\mu$ chosen as in \eqref{H2p}.
\begin{lemma}\label{Lemma:Expectation_estimate}
Under \eqref{H1} we have
 $\overline{U}_n=\sum_{k=1}^n\pi_k^n\,U_k$ and
$\EE(\|\overline{U}_n\|)\leq  \normeq\, \sumprob_n$.
\end{lemma}

\begin{proof}
The identity $\overline{U}_n=\sum_{k=1}^n\pi_k^n\,U_k$ is easily checked inductively, and then Jensen's inequality gives
 \[
 \EE(\|\overline{U}_n\|)\leq \normeq \sqrt{\EE(\|\overline{U}_n\|_2^2)}= \normeq \sqrt{\mbox{$\sum_{k=1}^n(\pi_k^n)^2\,\EE(\|U_k\|^2_2)$}}
 =  \normeq\,  \sumprob_n. 
 \]

 \end{proof}
 
Our first result on almost sure convergence requires the additional assumption
\begin{equation} \label{H2}\tag{\mbox{$\textsc{h}_2$}}
\begin{cases}
(a)&\sum_{k=1}^\infty \alpha_{k}\sumprob_{k-1}< \infty,\\[1ex]
(b)&\sum_{k=1}^\infty \alpha_k^2 \Var_k^2<\infty.
\end{cases}
\end{equation}
\begin{theorem}\label{Thm:First} 
Assume \eqref{H0}, \eqref{H1}, and \eqref{H2}.  
Then, the sequence  $(x_n)_{n\in\NN}$ generated by \eqref{Eq:skm} converges almost surely to some (random) point $x^* \in \fix(T)$.
\end{theorem}

\begin{proof} 
We claim that,  almost surely, the random variable $X_\infty\!\triangleq \!\sum_{k=1}^\infty\alpha_{k} \|\overline U_{k-1}\| $ is finite
 and $\overline U_n \to 0$. The convergence of $x_n$ then follows directly from 
  \cref{Thm:BCP} which gives $z_n\to x^*$ for some $x^*\in\fix(T)$, and 
therefore $x_n=z_n+\overline U_n\to x^*$.

To establish the previous claim, 
let $X_n=  \sum_{k=1}^{n} \alpha_{k} \|\overline U_{k-1}\|$.  From \cref{Lemma:Expectation_estimate}  we have $\EE( X_n)\leq \Delta$ with $\Delta = \normeq \sum_{k=1}^\infty \alpha_{k}\sumprob_{k-1}$, so that Lebesgue's monotone convergence theorem and \eqref{H2}(a) yield $\EE (X_\infty)= \lim_ {n \to \infty }\EE(X_n) \leq\Delta<\infty$, which implies that $X_\infty$ is finite almost surely. 

Next, we use a stochastic approximation argument to show that $\overline U_n\to 0$, for which
we rewrite the recursion \eqref{Eq:AveragedNoise} as
\begin{equation*}
  \overline U_{n+1} -\overline U_n = \alpha_{n+1} \left (- \overline U_n +U_{n+1}\right ).
\end{equation*}
We first show that $\|\overline{U}_n\|$ is bounded almost surely. 
Indeed, consider the martingale $M_{n} = \sum_{k=1}^{n} \alpha_{k} U_{k}$ which is bounded in $L^2$ with
\[
\EE \left (\norm{M_{n}}^2_2 \right )= \sum_{k=1}^{n} \alpha_{k}^2\, \EE \left (\norm{U_{k}}^2_2 \right ) \leq   \sum_{k=1}^{\infty} \alpha_{k}^2 \Var_k^2<\infty.
\]
Doob's martingale convergence theorem implies that $M_n$ converges to an almost surely finite random variable $M_\infty$, hence 
$\|M_n\|$ remains bounded. Now, from \eqref{Eq:AveragedNoise} we have 
\[
M_{n} = \sum_{k=1}^n\alpha_k U_k = \sum_{k=1}^n\overline U_{k} - (1-\alpha_{k} )\overline U_{k-1}=\overline U_{n} +\sum_{k=1}^n\alpha_{k}\overline U_{k-1}
\]
and then a simple triangle inequality gives
$$\|\overline U_{n} \|=\mbox{$\| M_{n} -\sum_{k=1}^n\alpha_{k}\overline U_{k-1}\|$}\leq \|M_n\|+X_\infty$$
which implies that $\|\overline{U}_n\|$ is bounded almost surely. 

To complete our proof that $\overline{U}_n\to 0$, we note that \cite[Proposition 4.2]{Benaim99} (applied with $q=2$ and $\gamma_n=\alpha_n$) 
remains valid under the assumption $\sum_{k=1}^\infty\alpha_k^2\EE(\|U_k\|_2^2)<\infty$ which is precisely \eqref{H2}(b).  This, combined 
with the Limit Set Theorem \cite[Theorem 5.7]{Benaim99} implies that the $\omega$-limit set of $(\overline U_n)_{n\in\NN}$ is almost surely an Internally 
Chain Transitive set of the differential equation $\dot U=-U$. Since $\{0\}\subseteq \RR^d$ is a global attractor, and hence the only ICT, 
we conclude that $\overline U_n \to 0$ almost surely as required.
\end{proof}

\begin{remark} Note that when $\fix(T)$ is not a singleton, the limit $x^*$ of the random sequence $(x_n)_{n\in\NN}$ in \cref{Thm:First},
is itself a random point in $\fix(T)$ which depends on the specific realization of the full noise sequence $(U_n)_{n\in\NN}$.
\end{remark}

\begin{example}\label{Ex:asc23}
For uniformly bounded variances $\Var_n^2\leq\Var^2$ and
stepsizes of power form $\alpha_n={1}/{(n+1)^a}$,
 \cref{Lemma:L3} in \cref{appendix1} gives $\sumprob_{n-1}\leq\theta\sqrt{\alpha_n}$
and therefore both \eqref{H0} and \eqref{H2} hold for all $a\in(\frac{2}{3},1]$.
Moreover, from \cite[Theorem 3.1.1]{chen} we obtain the
 almost sure convergence rate $n^b\|\overline U_n\|\to 0$
for all $b<a-\frac{1}{2}$. 
\end{example}

\subsection{Convergence of iterates for vanishing noise}

Another setting that ensures the convergence of the iterates $(x_n)_{n\in\NN}$ is when $\|U_n\|$ tends 
to zero in expectation compensating the divergence of the series $\sum_k\alpha_k$. Namely, suppose that 
\begin{equation*}\tag{\mbox{$\textsc{h}_3$}}\label{Eq:Vanishing_Noise}
\mbox{$\sum_{k=1}^\infty\alpha_k\,\escon{\|U_k\|}{\mathcal{F}_{k-1}}$ is finite almost surely.}
\end{equation*}
This excludes some simple situations covered by  \cref{Thm:First},
such as when the $U_n$'s are i.i.d., however it bypasses the need for \eqref{H1} and \eqref{H2}.
In particular it allows for nonvanishing $\alpha_k$'s, provided however that we assume the stronger 
condition that $\escon{\|U_k\|}{\mathcal{F}_{k-1}}$ is summable.

\begin{theorem}\label{Thm:Vanishing_Noise}
Assume \eqref{H0} and \eqref{Eq:Vanishing_Noise}. 
Then, the sequence  $(x_n)_{n\in\NN}$ generated by \eqref{Eq:skm} converges almost surely to some (random) point $x^* \in \fix(T)$. 
\end{theorem}
\begin{proof}
As in the proof of \cref{Thm:First}, it suffices to show that almost surely we have  $\overline{U}_{n}\to 0$ and 
$\sum_{k=1}^\infty\alpha_k\|\overline{U}_{k-1}\|<\infty$.

 Denote $N_n=\sum_{k=1}^n \alpha_k \|U_k\|$ and $\eta_n=\alpha_n\,\escon{\|U_n\|}{\mathcal{F}_{n-1}}$
 so that 
\[
\mbox{$\escon{N_{n}}{\mathcal{F}_{n-1}}=N_{n-1}+\alpha_{n}\,\escon{\|U_{n}\|}{\mathcal{F}_{n-1}}= N_{n-1}+\eta_{n}.$}
\] 
From \eqref{Eq:Vanishing_Noise} we have $\sum_{k=1}^\infty \eta_k<\infty$ and \cite[Theorem 1]{rs} implies that $N_n$ 
has a finite limit. Thus, the series $\sum_{k\geq 1}\alpha_k\,U_k$ is absolutely convergent,
and therefore its tail sum $u_n= \sum_{k> n}\alpha_{k}\,U_{k}$ is well defined and converges to 0. Adding $u_{n}$ on both sides of \eqref{Eq:AveragedNoise} and setting 
$w_n =  \overline U_{n} + u_{n}$ we get 
$$w_{n}=(1-\alpha_n)\,w_{n-1}+\alpha_n\, u_{n-1}$$
so that $w_n=\sum_{k=0}^n\pi_k^n\, u_{k-1}$. Now,  for $n\to\infty$ and $k$ fixed we have $\pi_k^n\to 0$  since
$$
\pi_k^n=\mbox{$\alpha_k\prod_{i=k+1}^n(1\!-\!\alpha_i)\leq \alpha_k\,\exp(-\sum_{i=k+1}^n\alpha_i)\to 0$.}
$$
Since $u_n\to 0$, using Toeplitz Lemma we get $w_n\to 0$ and therefore $\overline{U}_n\to 0$ as well.
Finally, from \eqref{Eq:AveragedNoise} we get $\|\overline{U}_{k}\|\leq(1-\alpha_{k})\|\overline{U}_{k-1}\|+\alpha_{k}\|U_{k}\|$ so that
$$
\mbox{$\sum_{k\geq 1}\alpha_k\|\overline{U}_{k-1}\|$}\leq \mbox{$\sum_{k\geq 1}(\|\overline{U}_{k-1}\|-\|\overline{U}_k\|+\alpha_k\|{U}_{k}\|) 
=\sum_{k\geq 1}\alpha_{k}\|U_{k}\|<\infty$,}
$$
and the convergence $x_n\to x^*\in\fix(T)$ follows as in the proof of  \cref{Thm:First}.
\end{proof}
Combining  \cref{Thm:First,Thm:Vanishing_Noise}  we readily get the following
\begin{corollary} Assume \eqref{H0}, \eqref{H2}, and suppose that 
$U_n=W_n+V_n$ where $W_n$ satisfies \eqref{H1} and $V_n$ satisfies \eqref{Eq:Vanishing_Noise}. 
Then, the sequence  $(x_n)_{n\in\NN}$ generated by \eqref{Eq:skm} converges almost surely to some (random) point $x^* \in \fix(T)$. 
\end{corollary}

 \cref{Thm:Vanishing_Noise} is similar to Corollary 2.7 in \cite[Combettes and Pesquet]{cp1}. The latter
 holds in infinite dimension but is restricted to separable Hilbert spaces. It
establishes the weak convergence 
$x_n\rightharpoonup x^*\in \fix(T)$ under the assumption
\begin{equation*}
\mbox{$\sum_{n= 1}^\infty\alpha_n\,\sqrt{\escon{\|U_n\|^2}{\mathcal{F}_{n-1}}}$ is finite almost surely}
\end{equation*}
which is slightly stronger  than \eqref{Eq:Vanishing_Noise} as follows from Jensen's inequality.

\subsection{Error bounds}\label{Sec:EB}
The previous analysis not only guarantees the almost sure convergence of the iterates $(x_n)_{n\in\NN}$
but can also be used to obtain estimates for the fixed-point residuals. Namely, since $I-T$ is 2-Lipschitz,
by considering the sequence $z_n=x_n-\overline U_n$ as before, we have
$$\|x_n-Tx_n\|\leq \|z_n-Tz_n\|+2\|z_n-x_n\|=\|z_n-Tz_n\|+2\|\overline U_n\|.$$
Now, \cref{Thm:First} 
gives conditions under which the $z_n$'s are 
bounded so there exists a finite $\kappa$ such that $\|Tz_n-x_0\|\leq\kappa$ for all $n\geq 1$, and therefore
\eqref{Eq:IKMbound} with $e_k=\overline U_{k-1}$ yields the following  pathwise estimate
\begin{equation}\label{Eq:IKMbound1}
\|x_n-Tx_n\|\leq \kappa\,\sigma(\tau_n)+\sum_{k=2}^n2\alpha_k\,\sigma(\tau_n\!-\!\tau_k)\|\overline{U}_{k-1}\|+4\|\overline{U}_{n}\|.
\end{equation}

Notice that $\overline{U}_{0}=0$ which explains why the term $k=1$ in the sum was omitted.

Unfortunately,  this 
 bound is not explicit since $\kappa$ is random and depends on the full realization of the trajectory. Moreover, the noise $U_n$ may occasionally produce large excursions
and $\|\overline{U}_n\|$ may become very large, so that \eqref{Eq:IKMbound1} does not give a
pathwise error bound with explicit constants. 
However,  \eqref{Eq:IKMbound1} combined with estimates for $\|\overline{U}_n\|$
(see {\em e.g.}  \cref{Ex:asc23}), still provides asymptotic rates that hold almost surely.
 
Hereafter we focus instead on the expected residuals $\EE(\|x_n-Tx_n\|)$, for which we obtain computable {\em non-asymptotic} error bounds.
Markov's inequality then yields explicit 
 error bounds that hold  pathwise with high probability. 
We will assume one of the following  (non-excluding) conditions and choices for the constant $\bar\kappa$
 \begin{equation}\tag{\mbox{$\textsc{h}_4$}}\label{Eq:BoundedKappa}
 \left\{\begin{array}{l}
\mbox{ $(i)$ either $T$ has bounded range and $\bar\kappa\geq\sup_{x\in\RR^d}\|Tx-x_0\|$,}\\[1ex]
\mbox{$(ii)$ or $S\triangleq\sum_{k=2}^\infty\alpha_k\sumprob_{k-1}<\infty$ and $\bar\kappa\geq 2\mathop{\rm dist}(x_0,\fix(T))+\mu\, S$. }
\end{array}\right.
\end{equation}

\begin{theorem}\label{Thm:Rate_Expectation}
Under \eqref{H1} and \eqref{Eq:BoundedKappa} we have
\begin{equation}\label{Eq:RateExpected}
\EE(\|x_n-Tx_n\|)\leq \bar\kappa\,\sigma(\tau_n)+2\normeq  \sum_{k=2}^n\alpha_k\,\sigma(\tau_n\!-\!\tau_k)\,\sumprob_{k-1}+4\normeq \, \sumprob_n.
\end{equation}
\end{theorem}
\begin{proof}
Under \eqref{Eq:BoundedKappa}\,(i) we have that \eqref{Eq:IKMbound1} holds with $\kappa=\bar\kappa$, and  
\eqref{Eq:RateExpected} follows by taking expectation in \eqref{Eq:IKMbound1}
and then using the estimate $\EE(\|\overline{U}_{k}\|)\leq\mu\,\sumprob_{k}$
in \cref{Lemma:Expectation_estimate}. 
For \eqref{Eq:BoundedKappa}\,(ii), we observe that  \eqref{Eq:IKMbound1} holds with $\kappa=2\mathop{\rm dist}(x_0,\fix(T))+\sum_{k=2}^\infty\alpha_k\|\overline{U}_{k-1}\|$. 
Using  \cref{Lemma:Expectation_estimate} we have $\EE(\kappa)\leq\bar\kappa$,
and then we get  \eqref{Eq:RateExpected} as before.
\end{proof}

In order to have a  more manageable error bound, we note that 
$\sumprob_n$ can  be sometimes estimated as a function of $\alpha_n$
which in turn can be bounded as a function of  $\tau_n$. For instance, if $\alpha_n=g(n)$ with 
$g(\cdot)$ positive and decreasing, then $\tau_n\approx  \int_1^ng(x)\,dx\triangleq G(n)$
and we get $\alpha_n\approx g(G^{-1}(\tau_n))$.
In such cases we may interpret the middle sum in \eqref{Eq:RateExpected}
as a Riemann sum for a certain convolution integral. The following result formalizes 
this general principle, which will be used later to study in detail the 
stepsizes of power form $\alpha_n={1}/{(n+1)^a}$.

\begin{theorem} \label{Thm:convolution_bound}
Assume  \eqref{H1}, \eqref{Eq:BoundedKappa}, and suppose there exists a 
decreasing convex function $h:(0,\infty)\to(0,\infty)$ of class $C^2$, and a constant $\gamma\geq 1$, such that
\begin{equation}\label{eq:convol}
(\forall k\geq 2)\quad\left\{
\begin{array}{cl}
(a)&\sumprob_{k-1}\leq(1\!-\!\alpha_k)\,h(\tau_k),\\[1ex]
(b)&\alpha_k(1\!-\!\alpha_k)\leq\gamma\,\alpha_{k+1}(1\!-\!\alpha_{k+1}).
\end{array}
\right.
\end{equation}
Then,  for all $n\geq 1$ we have 
\begin{equation}\label{Eq:RateExpected2}
\EE(\|x_n-Tx_n\|)\leq B_n\triangleq \frac{\bar{\kappa}}{\sqrt{\pi\tau_n}}+\frac{2 \mu\gamma}{\sqrt{\pi}}\! \int_{\tau_1}^{\tau_n}\!\!\!\frac{h(x)}{\sqrt{\tau_n\!-\!x}}\,dx
+2\mu\, \alpha_n\sumprob_{n-1}+4\mu \,\sumprob_n.
\end{equation}
\end{theorem}
\begin{proof}  
Consider the middle sum in \eqref{Eq:RateExpected}.  Since $\sigma(0)=1$ the $n$-th
term in the sum is $2 \normeq \,\alpha_n\sumprob_{n-1}$. Separating this last term,
and bounding the other ones using the inequalities $\alpha_k\sumprob_{k-1}\leq \alpha_k(1\!-\!\alpha_k)h(\tau_k)=h(\tau_k)(\tau_k\!-\!\tau_{k-1})$ 
and $\sigma(\tau_n\!-\!x)\leq{1}/{\sqrt{\pi(\tau_n\!-\!x)}}$, it suffices to show that
\begin{equation}\label{Eq:tbp}
 \sum_{k=2}^{n-1}\frac{h(\tau_k)}{\sqrt{\tau_n-\tau_k}}(\tau_k\!-\!\tau_{k-1})\leq \gamma \int_{\tau_1}^{\tau_n}\!\!\frac{h(x)}{\sqrt{\tau_n-x}}\,dx.
 \end{equation}

Let $f(x)=h(x)/\sqrt{\tau_n-x}$. We claim that there exists $y_n\in[\tau_1,\tau_n)$ such that $f$ decreases in the interval 
$[\tau_1,y_n]$ and increases in $[y_n,\tau_n]$. Indeed, the derivative is
$$f'(x)=\frac{h (x)+2 (\tau _n-x) h '(x)}{2 (\tau _n-x){}^{3/2}}$$
whose numerator $\psi(x)=h (x)+2 (\tau _n-x) h '(x)$ is non-decreasing
as one can check by computing $\psi'(x)=-h '(x)+2 (\tau _n-x) h ''(x)\geq 0$.
Since $\psi(\tau_n)=h (\tau_n)>0$, we either have $\psi(\tau_1)<0$ and there exists some 
$y_n\in (\tau_1,\tau_n)$ at which $\psi(\cdot)$, and thus also $f'$, changes from negative to positive,
or $\psi(\tau_1)\geq 0$ so that $f'(x)\geq 0$ for all $x\in [\tau_1,\tau_n]$ and we may just take $y_n=\tau_1$.

Let $k_n$ be the largest $k$ such that $\tau_k\leq y_n$.
Since $f(\cdot)$ decreases on $[\tau_1,y_n]$,
we can bound the partial sum up to $k_n$ on the left of \eqref{Eq:tbp} by the integral, namely
$$\sum_{k=2}^{k_n}f(\tau_k)(\tau_k-\tau_{k-1})\leq \int_{\tau_1}^{\tau_{k_n}}\!\!f(x)\,dx\leq  \gamma \int_{\tau_1}^{\tau_{k_n}}\!\!f(x)\,dx.$$
For the rest of the sum we note \eqref{eq:convol}\,(b) gives $(\tau_k\!-\!\tau_{k-1})\leq \gamma\,(\tau_{k+1}\!-\!\tau_{k})$ and since
 $f$ increases in $[y_n,\tau_n]$ we can again bound the sum by an integral, that is
 $$ \sum_{k=k_n\!+\!1}^{n-1}\!\!\!f\!(\tau_k)(\tau_k\!-\!\tau_{k-1})\leq  \gamma \sum_{k=k_n\!+\!1}^{n-1}\!\!\!\!f(\tau_k)(\tau_{k+1}\!-\!\tau_{k})\leq\gamma \int_{\tau_{k_n\!+\!1}}^{\tau_n}\!\!\!f(x)\,dx\leq\gamma \int_{\tau_{k_n}}^{\tau_n}\!\!\!f(x)\,dx.$$
Hence, \eqref{Eq:tbp} follows by summing both inequalities.
 \end{proof}

\vspace{2ex}
Markov's inequality readily gives error bounds that hold with high probability. 
\begin{corollary}\label{Cor:error_bound_high_prob} Under the assumptions of  \cref{Thm:convolution_bound}, for each 
$p\in (0,1)$ we have that the error bound $\|x_n-Tx_n\|\leq {B_n}/{p}$ holds with probability 
at least $1-p$.
\end{corollary}
\begin{proof} This follows directly from the estimate in  \cref{Thm:convolution_bound}
by taking $\epsilon=B_n/p$ in Markov's inequality  $\PP(\|x_n-Tx_n\|>\epsilon)\leq \EE(\|x_n-Tx_n\|)/\epsilon\leq B_n/\epsilon$.
\end{proof}

\section{An application to  \texorpdfstring{$Q$-learning}{Lg} for average reward MDPs}\label{sec:RVI-Q}
In this section we illustrate the previous results by an application to 
average reward Markov decision processes. As pointed out in \cite[Section 10.3]{SB2018}
and \cite{WNS2021},
this is arguably the most relevant setting for reinforcement learning, yet the results are far less
complete compared to the case of discounted rewards. Up to our knowledge, our results establish 
the first rate of convergence for the RVI-$Q$-learning iteration. 
We begin by recalling the general framework and known facts, which can be found in the classical reference \cite[Puterman]{Put2005}.
This book also discusses a large variety of applications.

Consider a controlled Markov chain with  state space $\SS$, action space
$\AA$, and transition probabilities $p(\cdot|i,u)\in\Delta(\SS)$, where $\SS$ and $\AA$  are finite sets and $p(j|i,u)$ is the probability 
 of moving to state $j\in\SS$ when we choose action $u\in\AA$ at state $i\in\SS$.
Let $g(i,u,j)\in\RR$ be a reward when moving from state $i$ to $j$ under action $u$.

Given an initial state $i_0\in\SS$ and a sequence of actions $(u_t)_{t\geq 1}$, the
chain evolves at random according to $i_{t+1}\sim p(\cdot|i_t,u_t)$. We look for
a stationary policy $\mu:\SS\to\AA$ such that the control $u_t=\mu(i_t)$ maximizes the 
long term expected reward 
\begin{equation}\label{eq:reward}
R_\mu(i_0)=\liminf_{T\to\infty}~\EE\left[\frac{1}{T}\sum_{t=0}^{T-1}g(i_t,u_t,i_{t+1})\right].
\end{equation}

By adding a constant to the rewards, we may assume that 
$0\leq g(i,u,j)\leq \overline{g}$ for some $\overline{g}\in\RR$, and therefore 
$0\leq R_\mu(i_0)\leq\overline{g}$. 
This translation does not
affect the optimal policies. Under the following standard ``unichain'' assumption
\begin{equation*}\tag{\mbox{\sc uni}}\label{Eq:Qlearn}
 \left\{\begin{array}{l}
\mbox{for every stationary policy $\mu$ the induced Markov chain  with 
 transition}\\
\mbox{probabilities $P_{i,j}=p(j|i,\mu(i))$ has a single recurrent class,}
\end{array}\right.
\end{equation*}
the $\liminf$ in \eqref{eq:reward} is attained as a limit and  $R_\mu(i_0)$
does not depend on the state $i_0$. Moreover,
the optimal value $\optv=\max_{\mu}R_\mu(i_0)$ satisfies $0\leq \optv\leq \overline{g}$ 
and is attained for some optimal stationary policy $\bar\mu$,
which can be found by solving the following system of Bellman equations
\begin{equation}\label{eq:Bell}
(\forall i\in\SS)\qquad V(i)=\max_{u\in\AA}\left[\sum_{j\in\SS}p(j|i,u)\big(g(i,u,j)+V(j)\big)-\optv\right],
\end{equation}
whose solution $V\in\RR^\SS$ is unique up to an additive constant.
Denoting  by $Q(i,u)$ the expression within square brackets, 
we have $V(i)=\max_{u\in\AA}Q(i,u)$, and an optimal policy is then obtained by setting $\bar\mu(i)=u$ where the action $u\in\AA$ maximizes $Q(i,\cdot)$.
We refer to \cite[Section 8.4]{Put2005} for details.

The system \eqref{eq:Bell}  can be expressed directly in terms of the $Q$-factors as
\begin{equation}\label{eq:Bell2}
(\forall (i,u)\in\SS\times\AA)\qquad Q(i,u)=\sum_{j\in\SS}p(j|i,u)\big (g(i,u,j)+\max_{u'\in\AA}Q(j,u')\big )-\optv.
\end{equation}
Let $d=|\SS\!\!\times\!\!\AA|$ and consider the map $M:\RR^d\to\RR^d$ given by
$$M(Q)(i,u)=\sum_{j\in\SS}p(j|i,u)\big(g(i,u,j)+\max_{u'\in\AA}Q(j,u')\big).$$
Denoting $e\in\RR^d$ the all-ones vector with $e(i,a)=1$ for all $(i,a)$, 
then \eqref{eq:Bell2} reduces to the fixed point equation $Q=H(Q)$ for the map $H(Q)=M(Q)-\optv\, e$,
whose solution is again unique up to an additive constant.

Note that $H$ is nonexpansive for $\|\cdot\|_\infty$.
However, in general the optimal value $\optv$ is 
unknown so one cannot directly implement  a fixed point iteration. Moreover, 
in many situations the transition probabilities are not known and one only has access to
samples from the distributions $p(\cdot|i,u)$. The 
 {\em RVI-$Q$-learning iteration} studied in
 \cite{abb1} 
addresses both issues by fixing a continuous function $f:\RR^d\to\RR$ such that $f(Q+c\,e)=f(Q)+c$
for all $c\in\RR$ and then, starting from an initial guess $Q_0\in\RR^d$, computes the recursive updates
\begin{equation}\label{eq:Q-RVI}\tag{\mbox{\sc rvi-q}}
Q_{n}(i,u)\!=\!(1\!-\!\alpha_{n})Q_{n\!-\!1}(i,u)\!+\!\alpha_n\big[g(i,u,\xi^{n}_{iu})\!+\!\max_{u'\in\AA}Q_{n\!-\!1}(\xi^{n}_{iu},u')\!-\!f(Q_{n\!-\!1})\big].
\end{equation}
with independent samples $\xi_{iu}^{n}\sim p(\cdot|i,u)$ for each pair $(i,u)\in\SS\times\AA$. 
Simple examples for $f$ are the max function $f(Q)=\max_{i,u}Q(i,u)$, the average $f(Q)=\frac{1}{d}\sum_{i,u}Q(i,u)$, and $f(Q)=Q(i_0,u_0)$ for a fixed $(i_0,u_0) \in \SS \times \AA$.

There are two main difficulties to analyze the iteration \eqref{eq:Q-RVI}. Firstly, while 
$H$ is nonexpansive, this property is lost when we introduce the stabilizer $f(Q)$.
To deal with this issue we will study an hypothetical iteration where instead of $f(Q_{n-1})$ we 
consider the unknown value $\optv$: convergence of \eqref{eq:Q-RVI} will  follow from
the convergence of the latter. The second difficulty comes from the fact that the variances $\Var_n^2$ 
might increase with $n$. However, we will show that they grow in a 
controlled manner which can still be handled with our general results. For clarity we address both
issues separately. We first establish a general consequence of \cref{Thm:First,Thm:convolution_bound}
with possibly increasing variances, and  later we show how it applies to \eqref{eq:Q-RVI}.

\begin{theorem}\label{Thm:RVI-Qg} 
Consider the iteration \eqref{Eq:skm} for $\alpha_n=1/(n+1)^a$ with  $\frac{4}{5}<a\leq 1$.
Assume \eqref{H1} and suppose in addition that  there exist constants $c$ and $\eta$ such that
\begin{equation}\label{eq:vargrowth}
(\forall n\geq 1)\quad\left\{
\begin{array}{cl}
(a)&\|x_n\|\leq\|x_{n-1}\|+\eta\,\alpha_n\mbox{ almost surely, and}\\[1ex]
(b)&\EE(\|U_n\|^2)\leq c(1+\EE(\|x_{n-1}\|^2)).
\end{array}
\right.
\end{equation}
Then,  the sequence $(x_n)_{n\in\NN}$ generated by \eqref{Eq:skm} converges almost surely to some fixed point $x^*\in\fix(T)$,
and there exists a constant $\kappa_a$ such that
\begin{equation}\label{eq:Q-RVI-Errorg}
\EE(\|x_n-Tx_n\|)\leq\frac{\kappa_a}{\sqrt{\tau_n}}\sim
\left\{\begin{array}{cl}
    O(1/\sqrt{n^{1-a}})&\mbox{if $\frac{4}{5}< a<1$},\\
O(1/\sqrt{\log n})&\mbox{if $a=1$.}
\end{array}\right.
\end{equation}
\end{theorem}
\begin{proof} The proof combines \cref{Thm:First} and \cref{Thm:convolution_bound}.

\vspace{1ex}
\noindent\underline{\em Convergence}.
 Since \eqref{H0} and \eqref{H1} are clearly satisfied, in order to apply \cref{Thm:First}
 it suffices to check  \eqref{H2}.
From \eqref{eq:vargrowth}(a) we get $\|x_n\|\leq\|x_0\|+\eta\sum_{k=1}^{n}\alpha_k\leq\tilde\eta\,\tau_n$ for
some constant $\tilde\eta$. Then, using  \eqref{eq:vargrowth}(b)
we can find another constant $\zeta$ such that
$$\theta_n^2=\EE(\|U_n\|_2^2)\leq\zeta^2\tau_{n-1}^2<\infty.$$ 
Since $\tau_n$ is increasing, from \cref{Lemma:L3} in \cref{appendix1} we get
$$\sumprob_n= \sqrt{\mbox{$\sum \nolimits_{k=1}^n(\pi_k^n)^2\Var_k^2$}}\leq \zeta\tau_{n-1}\sqrt{\mbox{$\sum_{k=1}^n(\pi_k^n)^2$}}\leq \zeta\tau_n\sqrt{\alpha_{n+1}},$$
and therefore
\begin{equation} \label{H2-RVI-betag}
\begin{cases}
(a)&\sum_{k=1}^\infty \alpha_{k}\sumprob_{k-1}\leq S_1\triangleq \zeta\sum_{k=1}^\infty \alpha_{k}^{3/2}\tau_{k-1},\\[1ex]
(b)&\sum_{k=1}^\infty \alpha_k^2 \Var_k^2\leq S_2\triangleq \zeta^2\sum_{k=1}^\infty \alpha_k^2\tau_{k-1}^2.
\end{cases}
\end{equation}
When $a<1$ we have $\tau_{k-1}\sim k^{1-a}$ while  for $a=1$ we get $\tau_{k-1}\sim \ln k$.
It follows that for $\frac{4}{5}<a\leq 1$ both series $S_1$ and $S_2$ in \eqref{H2-RVI-betag} converge, and \eqref{H2} holds.
Invoking \cref{Thm:First} we conclude that $x_n$ converges almost surely to some $x^*\in\fix(T)$.

\vspace{1ex}
\noindent\underline{\em Error bound}. We next show that $\EE(\|x_n-Tx_n\|)\leq B_n\sim O(1/\sqrt{\tau_n})$ using \cref{Thm:convolution_bound}. 
Note that the assumptions are satisfied, including
\eqref{Eq:BoundedKappa}(ii) which holds for $\bar\kappa = 2\|x_0-x^*\|+\mu\, S_1$ with $S_1$  as in \eqref{H2-RVI-betag}(a),
and with the finite constant
 \begin{equation}\label{eq:gamma_ag}
 \gamma_a=\mbox{$\sup_{k\geq 2}\frac{\alpha_k(1-\alpha_k)}{\alpha_{k+1}(1-\alpha_{k+1})}~=~\sup_{k\geq 2}(\frac{k+2}{k+1})^{2a}\frac{(k+1)^a-1}{(k+2)^a-1}$}
\end{equation}
which is slightly larger than 1 and is maximal for $a=1$ so that $\gamma_a\leq \gamma_1=\frac{32}{27}$.

From $\sumprob_{n}\leq\zeta \tau_{n}\sqrt{\alpha_{n+1}}$ it follows that
the two last terms $2\alpha_n\sumprob_{n-1}+4\sumprob_n$ in $B_n$ have an asymptotic order
$o(1/\sqrt{\tau_n})$ and are dominated by the first term $\bar\kappa/\sqrt{\tau_n}$.
For the convolution integral in $B_n$ we need $h(\cdot)$ such that $\sumprob_{k-1}\leq(1-\alpha_k)h(\tau_k)$.
Combining  the bound $\sumprob_{k-1}\leq\zeta \tau_{k}\sqrt{\alpha_k}$ with Lemmas \ref{Lemma:kappa} and 
\ref{Lemma:alfa1}, we may find a constant $\omega$ such that this holds for $h(x)=\omega\,x\,(1+x)^{-b}$ 
with $b={a}/{2 (1\!-\!a)}$ if $a<1$, and for $h(x)=\omega\, x\,e^{-x/2}$ when $a=1$.
These functions are not  convex and decreasing as required by \cref{Thm:convolution_bound}, 
but these properties hold in some interval $[x_a,\infty)$. By suitably modifying
$h(x)$ on $[0,x_a]$ we can ensure that these conditions are met everywhere. 
Splitting the convolution integral, one can easily check that the integral over $[0,x_a]$ is 
of order $O(1/\sqrt{\tau_n})$, while fixing $\rho\in (0,1)$ and noting that $b>2$,
the integral over $[x_a,\tau_n]$ can be estimated as 
\begin{eqnarray*}
\int_{x_a}^{\tau_n}\!\!\!\frac{h(x)}{\sqrt{\tau_n-x}}\,dx&\leq&\int_{x_a}^{\rho\tau_n}\!\!\!\frac{h(x)}{\sqrt{\tau_n-x}}\,dx+\int_{\rho\tau_n}^{\tau_n}\!\!\frac{h(x)}{\sqrt{\tau_n-x}}\,dx\\
&\leq&\frac{1}{\sqrt{(1\!-\!\rho)\tau_n}}\int_{x_a}^{\infty}\!\!\!h(x)\,dx+\int_{\rho\tau_n}^{\tau_n}\frac{h(\rho\tau_n)}{\sqrt{\tau_n-x}}\,dx
\end{eqnarray*}
which is again of order $O(1/\sqrt{\tau_n})$. Thus, all four terms in $B_n$ are of order $O(1/\sqrt{\tau_n})$
and we may find some constant $\kappa_a$ satisfying \eqref{eq:Q-RVI-Errorg}.
\end{proof}

We now apply \cref{Thm:RVI-Qg} to study the RVI-$Q$-learning iteration.
\begin{theorem}\label{Thm:RVI-Q2} Assume \eqref{Eq:Qlearn} and let $\alpha_n=1/(n+1)^a$ with  $\frac{4}{5}<a\leq 1$. 
Then, almost surely we have  $f(Q_n)\to\optv$ and $Q_n\to Q^*\in\fix(H)$ where $Q^*$ is the unique fixed point satisfying  
$f(Q^*)=\optv$. Moreover, there exists a constant $\kappa_a$ such that
\begin{equation}\label{eq:Q-RVI-Error2}
\EE(\|Q_n-H(Q_n)\|_{\infty})\leq\frac{\kappa_a}{\sqrt{\tau_n}}\sim
\left\{\begin{array}{cl}
    O(1/\sqrt{n^{1-a}})&\mbox{if $\frac{4}{5}< a<1$},\\
O(1/\sqrt{\log n})&\mbox{if $a=1$.}
\end{array}\right.
\end{equation}
\end{theorem}
\begin{proof}
We will compare  $(Q_n)_{n\in\NN}$  to  an hypothetical  sequence $(Q_n^\optv)_{n\in\NN}$
generated from the same initial point $Q_0^\optv=Q_0$  by the recursion
\begin{equation}\label{eq:Q-RVI-beta2}
Q_{n}^{\optv}(i,u)=(1-\alpha_{n})Q^{\optv}_{n-1}(i,u)+\alpha_n\big [g(i,u,\xi^{n}_{iu})+\max_{u'\in\AA}Q^{\optv}_{n-1}(\xi^{n}_{iu},u')-\optv\big ].
\end{equation}
This cannot be implemented (unless $\optv$ is known), but we can still use it as a benchmark 
to compare with the original $Q_n$.  
 In order to study $Q_n^\optv$, we consider the filtration ${\mathcal F}_n=\sigma(\{(Q^\optv_n,\xi^k): k\leq n\})$ induced by the process,
and denote $Z_{n}\in\RR^d$ the random vector with components
 $$Z_{n}(i,u)=g(i,u,\xi^{n}_{iu})+\max_{u'\in\AA}Q^\optv_{n-1}(\xi^{n}_{iu},u').$$ 
Observing that $Q^\optv_{n-1}$ is ${\mathcal F}_{n-1}$-measurable,
we have $\EE(Z_{n}|{\mathcal F}_{n-1})=M(Q^\optv_{n-1})$. Hence, $U_{n}=Z_{n}-M(Q^\optv_{n-1})$ is a martingale 
difference sequence and \eqref{eq:Q-RVI-beta2} can be rewritten in compact form in the framework of \eqref{Eq:skm} as
\begin{equation*}
Q_{n}^{\optv}=(1-\alpha_{n})Q^{\optv}_{n-1}+\alpha_n\big(H(Q^\optv_{n-1})+U_{n}\big).
\end{equation*}
Now, \eqref{H0} is clearly satisfied, while  $\EE(U_n|{\mathcal F}_{n-1})=0$ holds by construction.
On the other hand, from the definition of $U_n$ it follows directly that $\|U_n\|_{\infty}\leq \overline{g}+2\|Q^\optv_{n-1}\|_{\infty}$.
Also, since $|g(i,u,\xi_{iu}^{k})-v|\leq\overline{g}$, from \eqref{eq:Q-RVI-beta2} we get 
$\|Q^\optv_n\|_{\infty}\leq\|Q^\optv_{n-1}\|_{\infty}+ \overline{g}\,\alpha_n$.

This shows that  all the conditions of \cref{Thm:RVI-Qg} are met. It follows that 
$Q^\optv_n$ converges  almost surely to some $Q^\optv\in\fix(H)$, and
{\em a fortiori} $f(Q^\optv_{n})\to f(Q^\optv)$.
Now, substracting \eqref{eq:Q-RVI} and \eqref{eq:Q-RVI-beta2}, a simple induction shows that 
$Q_n-Q_n^\optv=c_n\, e$ where the sequence $c_n\in\RR$ satisfies $c_0=0$ and 
\begin{equation*}
 c_n=(1-\alpha_n)c_{n-1}+\alpha_n\big(\optv-f(Q^\optv_{n-1})),
 \end{equation*}
 which implies $c_n\to\optv-f(Q^\optv)$. It follows that $Q_n\to Q= Q^\optv+(\optv-f(Q^\optv))\,e$
which is also in $\fix(T)$ with $f(Q)=\optv$, so that in fact $Q=Q^*$ and  $f(Q_n)\to\optv$.
Finally, since $Q\mapsto Q-H(Q)$ is invariant by addition of constants, we have $\|Q_n-H(Q_n)\|_\infty=\|Q^\optv_n-H(Q^\optv_n)\|_\infty$ 
and \eqref{eq:Q-RVI-Error2} follows directly from \eqref{eq:Q-RVI-Errorg} applied to $Q_n^\optv$.
\end{proof}

\begin{remark} With a more detailed analysis, similar to the one  in \cref{appendix00},
one can obtain an explicit expression for $\kappa_a$
in terms of the primitives $a$, $\overline{g}$ and $Q_0$. 
\end{remark}

\subsection{Numerical example} 
We illustrate \cref{Thm:RVI-Q2} on a small example borrowed from \cite[Duff]{Duff1995},
with two states $\SS=\{1,2\}$ and two actions $\AA=\{u_1,u_2\}$. The rewards depend only on the 
final state with $g(i,u,1)=1.0$  and $g(i,u,2)=10.0$ for all $(i,u) \in \SS \times\AA$, and the 
actions only control the transition probabilities as in \cref{fig:transprob}.
\begin{figure}[ht] 
\centering
\begin{tikzpicture}[thick,scale=1.2,->,shorten >= 1pt,shorten <= 1pt,every node/.style={scale=0.6}]
 \node[draw,circle] (1) at (0,0) {$1$};
  \node[circle] (3) at (1.5,0.05) {\large action $u_1$};
 \node[draw, circle] (2) at (3,0) {$2$};
 \draw (1) to [->,out=135,in=225,looseness=8] node [fill=white] {0.3} (1);
\draw (1) to [bend left = 35]   node [fill=white] {0.7}  (2) ;
\draw (2) edge [bend left]   node  [fill=white]  { 0.7} (1) ;
\draw (2) edge [out=45,in=-45,looseness=8]   node  [fill=white]  {0.3} (2) ;
\end{tikzpicture}
\hspace{1cm}
\begin{tikzpicture}[thick,scale=1.2,->,shorten >= 1pt,shorten <= 1pt,every node/.style={scale=0.6}]
 \node[draw,circle] (1) at (0,0) {1};
 \node[draw, circle] (2) at (3,0) {2};
 \draw (1) to [->,out=135,in=225,looseness=8] node [fill=white] {0.9} (1);
\draw (1) to [bend left = 35]   node [fill=white] {0.1}  (2) ;
  \node[circle] (3) at (1.5,0.05) {\large action $u_2$};
\draw (2) edge [bend left]   node  [fill=white]  {0.1} (1) ;
\draw (2) edge [out=45,in=-45,looseness=8]   node  [fill=white]  {0.9} (2) ;
\end{tikzpicture}
\vspace{-1ex}

\caption{Transition probabilities for actions $u_1$ (left) and $u_2$ (right)\label{fig:transprob}}
\end{figure}
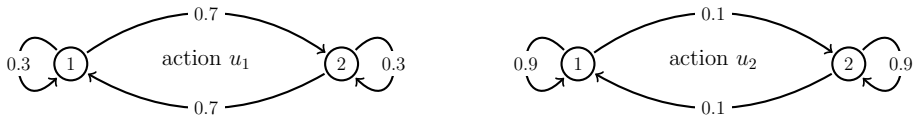

\noindent
Observe that this example clearly satisfies \eqref{Eq:Qlearn}. The optimal value is $v=\frac{71}{8}$ and, for $f(Q)=\max_{i,u}Q(i,u)$,  the fixed point $Q^*$ is 
$$Q^*=\frac{1}{8}\begin{pmatrix} 53 &-1\\ 17&71\end{pmatrix}.$$ 
We performed 100 independent runs of RVI-$Q$-learning starting from $Q_0\equiv 0$ and with  $\alpha_n=1/(n+1)^a$ (10.000 iterations in each run).
The upper plots in \cref{Fig:RVI-fQ} show the evolution of the values $f(Q_n)$ over 10 runs, for  $a=1$ (left) and  $a=\frac{4}{5}$ (right).
The bottom plots show the corresponding rates for $\EE(\|Q_n-H(Q_n)\|_\infty)$, estimated as the average over the 100 runs and rescaled
by the theoretical rates.
These graphics are in agreement with our estimates, although the decreasing trend  in the lower  plots suggest that there might be some room for
improving our asymptotic rates. The trajectories for $a=\frac{4}{5}$, which in fact is not covered by \cref{Thm:RVI-Q2},
are more noisy than the simulations for $a=1$. Further simulations with $a$ close to $\frac{1}{2}$ exhibit larger and persistent oscillations, 
even when started from the solution $Q_0=Q^*$, with no signs of achieving convergence.

\begin{figure}[htbp]
    \centering
\includegraphics[scale=0.9,clip, trim=4.5cm 10.5cm 1.3cm 4cm]{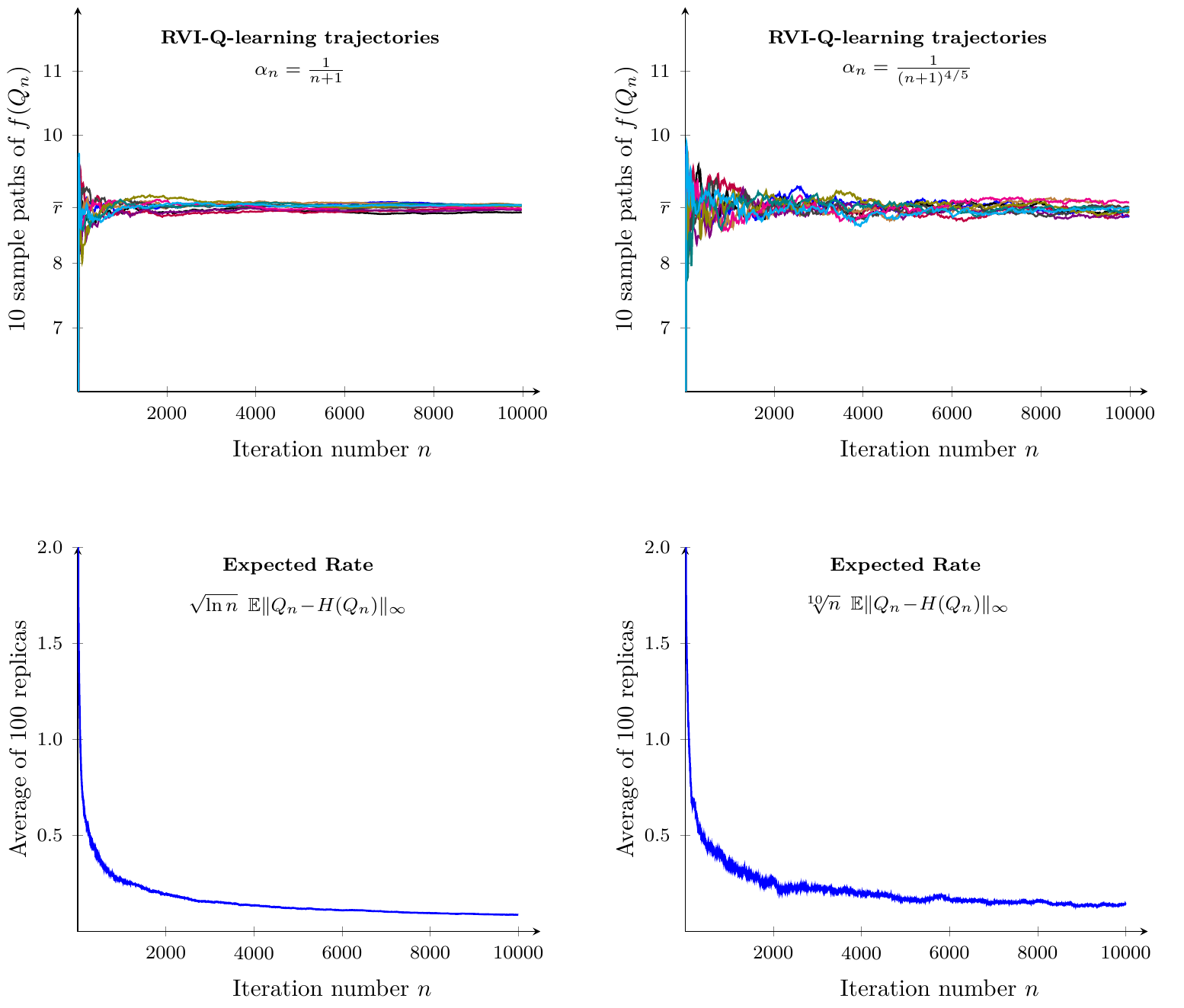} 
    \caption{\label{Fig:RVI-fQ} RVI-$Q$-learning for $\alpha_n=1/(n+1)^a$ with $a=1$ (left) and $a=\frac{4}{5}$ (right).}
\end{figure}

\section{The case of uniformly bounded variances}\label{sec:BndVar}
In this section we derive more explicit error bounds for the case where the noise sequence $U_n$
is a martingale sequence which is uniformly bounded in $L^2$, that is, under the assumption
 \begin{equation*}\tag{\mbox{\sc bVar}}\label{Eq:BoundedVar}
\mbox{There  exists a finite constant $\Var\geq 0$ such that $\sup_{n\in\NN}\Var_n\leq\Var$. }
\end{equation*}
We analyze subsequently 
the case of constant stepsizes over a fixed horizon, and then the case of stepsizes of power form.

\subsection{Fixed horizon with constant stepsizes}

By considering a pre-fixed number of iterations $n_0$ we can still exploit the 
bound \eqref{Eq:RateExpected} with constant stepsizes $\alpha_n\equiv\alpha$.

\begin{theorem}\label{Thm:Rate_Expectation_constant}
Assume \eqref{H1}, \eqref{Eq:BoundedKappa} and  \eqref{Eq:BoundedVar}.
For the \eqref{Eq:skm} iterates with constant
stepsizes $\alpha_n\equiv\alpha$ and $\beta=1-\alpha$, we have
\begin{equation}\label{Eq:error_constant_stepsize}
\EE(\|x_n-Tx_n\|)\leq \frac{ \bar{\kappa}}{\sqrt{\pi\alpha\beta\,n}}+\frac{2\normeq \Var \sqrt{\alpha}}{\sqrt{1\!+\!\beta}}\left[2\sqrt{\frac{\alpha\, n}{\pi\beta}}+\alpha+2\right].
\end{equation}
In particular, considering a fixed $n_0$ and constant stepsizes $\alpha_n\equiv{1}/{(6\, n_0^{2/3})}$ 
we have
\begin{equation}\label{Eq:rate_constant_stepsize}
\EE(\|x_{n_0}-Tx_{n_0}\|)
\leq \frac{2(\bar\kappa+\mu\Var)}{\sqrt[6]{n_0}}.
\end{equation}
\end{theorem}

\begin{proof}
With constant stepsizes we have $\tau_n=\alpha\beta\,n$ and $\pi_k^n=\alpha\beta^{n-k}$ so that  
$$\sumprob_n\leq\mbox{$\Var\alpha\sqrt{\sum_{k=1}^n\beta^{2(n-k)}}\leq \frac{\Var\alpha}{\sqrt{1-\beta^2}}=\frac{\Var\sqrt{\alpha}}{\sqrt{1+\beta}}$},$$
and then \eqref{Eq:error_constant_stepsize} follows directly from \eqref{Eq:RateExpected2} with the constant function $h(x)\equiv\frac{\Var\sqrt{\alpha}}{\beta\sqrt{1+\beta}}$.
Substituting $\alpha={1}/{(6\, n_0^{2/3})}$ in \eqref{Eq:error_constant_stepsize} and after some routine algebra we get
$$\EE(\|x_{n_0}-Tx_{n_0}\|)\leq \frac{1}{\sqrt[6]{n_0}}\left(\mbox{$\frac{\sqrt{6}\, \bar\kappa }{\sqrt{\pi (1-\alpha)}}+\left(\frac{2+\alpha}{\sqrt{6}\sqrt[6]{n_0}}+\frac{1}{3\sqrt{\pi (1-\alpha)}}\right)  \frac{2\mu\Var}{\sqrt{2-\alpha}}$}\right)$$
and \eqref{Eq:rate_constant_stepsize} follows using $\alpha\leq \frac{1}{6}$ and $n_0\geq 1$ to bound the coefficients of 
$\bar\kappa $ and $\mu\Var$ in the expression in parenthesis.
\end{proof}

\begin{remark}
The choice $\alpha_n\equiv{1}/{(6\, n_0^{2/3})}$ is just one among many alternatives. As a mater of fact,
for any given values of $(\bar\kappa,\Var,\mu, n_0)$ the bound \eqref{Eq:error_constant_stepsize} has a unique 
minimizer $\alpha^*\in(0,1)$ which can be computed numerically and yields a
nonasymptotic bound $O({1}/{\sqrt[6]{n_0}})$ with a smaller  proportionality constant. 
\end{remark}

\begin{remark} \label{remark3.6} A natural question is whether the error bound \eqref{Eq:error_constant_stepsize}
can be improved by using some variance reduction technique. The simplest 
 approach would be to use a minibatch of size $k$ at every iteration, which reduces the variance 
 from $\Var^2$ to $\Var^2/k$, and then choose $k$ 
 and $\alpha$ that jointly minimize the bound for a fixed horizon $n_0$. In terms of oracle complexity, 
it turns out that for a given budget $N=k \,n_0$ of calls to the oracle
  the optimal batch size is $k=1$, which is just the original iteration.
We also tried variable batch sizes but were unable to improve the  complexity.
This confirms that \eqref{Eq:skm} has already some built-in variance reduction.
Naturally, this does not exclude the possibility of an improvement by using more sophisticated 
variance reduction schemes, or alternative bounds stronger than those in \cref{Thm:BCP}.
\end{remark}

\subsection{Stepsizes of power form}
Let us use \eqref{Eq:RateExpected2} to obtain  explicit error bounds and rates of convergence for 
stepsizes of the form $\alpha_n={1}/{(n+1)^a}$ with $\frac{1}{2}\leq a\leq 1$. We will see that the 
last two terms  $2 \normeq \Var \,\alpha_n\sumprob_{n-1}+4 \normeq \Var \,\sumprob_n$ in $B_{n}$ are
asymptotically dominated by the first two, and that each one of the leading terms may
dominate depending on the exponent $a$. For $a=\frac{2}{3}$ both terms are of the same 
order and yield the fastest rate which turns out to be of order $O(\ln(n)/\sqrt[6]{n})$.

To state  precisely our next result we denote 
\begin{equation}\label{eq:taua}
\tau_n(a)=\sum_{k=1}^n\alpha_k(1-\alpha_k)=\sum_{k=1}^n\mbox{$\frac{1}{(k+1)^a}(1-\frac{1}{(k+1)^a})$}
\end{equation}
and, in addition to $\gamma_a$ given by \eqref{eq:gamma_ag}, we
introduce the following parameters
\begin{equation}\label{eq:power}
\left\{
\begin{array}{rcl}
b_a&=&\mbox{$\frac{a}{2(1-a)}$}\\[0.5ex]
\lambda_a&=&\mbox{$\big(3^{a/2}-3^{-a/2}\big)^{2(1-a)/a} -(1\!-\!a)\, \tau_2(a)$}\\[0.5ex]
d_a&=&\mbox{$\frac{2\gamma_a}{3a-2}\big(\lambda_a+(1-a) \,\tau _1(a) \big)^{1-b_a}$}.
\end{array}
\right.
\end{equation}

\begin{theorem}\label{Thm:RatePower} Assume \eqref{H1}, \eqref{Eq:BoundedKappa},  \eqref{Eq:BoundedVar}, and $\alpha_n\!=\!{1}/{(n\!+\!1)^a}$ with $\frac{1}{2}\!\leq a\!\leq 1$.
Then $\EE(\|x_n-Tx_n\|)\leq\! B_{n,a}$ with explicit error bounds $B_{n,a}$  (see
\cref{Cor:ErrorBoundPower,Cor:ErrorBoundPower1,le:B3}), whose tight asymptotics as $n\to\infty$ are given by
\begin{equation}\label{Eq:ratefinal}
B_{n,a}\approx\left\{\begin{array}{ll}
\dfrac{2\normeq \Var\gamma_a}{\sqrt{1-a}}\dfrac{\Gamma(1-b_a)}{\Gamma(\frac{3}{2}-b_a)}\, \dfrac{1}{n^{a-\frac{1}{2}}}&\mbox{if }\frac{1}{2}\leq a<\frac{2}{3},\\[3ex]
 \dfrac{2\normeq \Var\gamma_{2/3}}{\sqrt{3\pi}}\,\dfrac{\ln(n)}{\sqrt[6]{n}}&\mbox{if }a=\frac{2}{3},\\[3ex]
 \dfrac{(\bar\kappa+2\normeq \Var d_a )\sqrt{1-a}}{\sqrt{\pi}}\,\dfrac{1}{n^{(1-a)/2}}&\mbox{if }\frac{2}{3}<a<1,\\[3ex]
 \dfrac{\bar\kappa+2\mu\Var d_1}{\sqrt{\pi}}\dfrac{1}{\sqrt{\ln n}}&\mbox{if }a=1. 
\end{array}\right.
\end{equation}
Here, $\Gamma$ stands for the Gamma function.
\end{theorem}
\begin{proof} The proof is technical and is presented in \cref{appendix00}.
It proceeds by analyzing separately the different ranges of the exponent $a$.
\end{proof}

Notice that when $a$ approaches $\frac{1}{2}$ the convergence rate slows down, and 
for $a=\frac{1}{2}$ the error bound does not converge to 0 but rather it has a strictly positive limit $2\sqrt{2\pi}\,\gamma_{1/2}\,\normeq \Var\sim 5.16\,\normeq \Var$.
\subsection{An application in stochastic optimization}\label{subsec:SGD}

Consider the standard stochastic gradient descent iteration for minimizing an expected cost $f(x)=\EE (F(x,\xi))$, 
with $F:\RR^d\times\Xi\to\RR$ and $\xi$ a random variable with values in $\Xi$, namely
\begin{equation}\label{Eq:sgd2}
\tag{\textsc{sgd}}
x_{n}=x_{n-1}-\gamma_n\nabla_x F(x_{n-1},\xi_n)
\end{equation}
where the $\xi_n$'s are i.i.d. samples of the random variable $\xi$. 

We assume that $f(\cdot)$ is
a well-defined smooth convex function, and  that $F(\cdot,\xi)$ is differentiable with $\nabla f(x)=\EE(\nabla_x F(x,\xi))$ globally
Lipschitz with a known constant $L$ for $\|\cdot\|_2$. 
As noted in \S\ref{sec:related},  \eqref{Eq:sgd} coincides 
with \eqref{Eq:skm} for $Tx=x-\frac{2}{L}\nabla f(x)$, $\alpha_n=\gamma_n L/2$,
 and $U_n=\frac{2}{L}[\nabla f(x_n)-\nabla_x F(x_n,\xi_n)]$.
 Clearly $\fix(T)=\mathop{\rm Argmin} f$, which we assume nonempty,
 while \cite[Corollary 10]{bh} shows that $T$ is nonexpansive. 
 
 \begin{corollary} In the above \eqref{Eq:sgd2} setting, take
 $\gamma_n=\frac{2}{L}\frac{1}{(n+1)^a}$ with $\frac{1}{2}\leq a\leq 1$
and suppose that the noise has uniformly bounded variance $\EE(\|U_n\|_2^2)\leq\Var^2$.
Then,
  \begin{equation}\label{eq:gradnorm}
(\forall n\geq 1)\qquad  \EE(\|\nabla f(x_n)\|_2\big)\leq \mbox{$\frac{L}{2}$} B_{n,a},
  \end{equation}
  with $B_{n,a}$ as in \cref{Thm:RatePower}. Moreover,
if $a\in(\frac{2}{3}, 1]$  then the iterates $(x_n)_{n\in\NN}$ converge  almost 
surely towards some point in $\mathop{\rm Argmin} f$.
 \end{corollary}
 \begin{proof} The bound \eqref{eq:gradnorm} follows directly from   \cref{Thm:RatePower},
 while almost sure convergence for $a>\frac{2}{3}$ is a consequence of
 \cref{Thm:First} (see also \cref{Ex:asc23}).
 \end{proof}

 As far as we know,  almost  sure convergence of the \eqref{Eq:sgd2} iterates was only proved recently 
 in \cite[Liu and Yuan]{LY18} for $a\in(\frac{2}{3},1)$ and under a weak assumption on the variances 
 ({\em cf.}  \cite[Condition ABC]{LY18}). In the special case of bounded variances, 
 we recover this from \cref{Thm:First},  including also the limit case $a=1$. 
 
 In addition, \cite{LY18} established almost sure asymptotic rates for the optimality gap 
 $f(x_n)-\min f$. As mentioned in \S\ref{Sec:EB}, the pathwise estimate \eqref{Eq:IKMbound1} 
 can also be used to obtain almost sure rates for the gradient norm $\|\nabla f(x_n)\|_2$.
 However, in contrast with the
 error bound in expectation \eqref{eq:gradnorm}, these almost sure rates involve stochastic multiplicative constants 
 for which in general there is no {\em a priori} control, so they do not provide explicit error bounds. In this sense, we note that Markov's inequality
 and \eqref{eq:gradnorm} readily give error bounds that hold with high probability.

 In order to put  \eqref{eq:gradnorm} in perspective, 
 we recall that for stepsizes $\gamma_n= c/n^{a}$ with  $c$ constant,
Bach and Moulines obtained in \cite[Theorem 4]{bm} explicit error bounds for the expected optimality 
gap $\EE\big(f(x_n)-\min f\big)\leq b_{n,a}$. This readily yields
an alternative bound for $\EE(\|\nabla f(x_n)\|_2)$. Indeed, a convex function with $L$-Lipschitz gradient satisfies 
\begin{equation*}
f(x)+\langle\nabla f(x),y-x\rangle+\mbox{$\frac{1}{2L}$}\|\nabla f(y)-\nabla f(x)\|_2^2\leq f(y)\qquad(\forall x,y\in\RR^d),
\end{equation*}
so that taking $y=x_n$ and $x\in\mathop{\rm Argmin} f$ we get $\|\nabla f(x_n)\|_2^2\leq 2L(f(x_n)-\min f)$,
and  Jensen's inequality then yields
 \begin{equation}\label{eq:gradnorm2}
  \EE(\|\nabla f(x_n)\|_2\big)\leq \sqrt{2L\,b_{n,a}}.
  \end{equation}
  
It is noteworthy that both \eqref{eq:gradnorm} and \eqref{eq:gradnorm2} give exactly the same asymptotic rates, 
except for $a=\frac{2}{3}$ where $\sqrt{b_{n,a}}\sim 1/\sqrt[6]{n}$ 
 and $B_{n,a}\sim \log n/\sqrt[6]{n}$ differ by a log factor. 
The coincidence in the asymptotic rates is remarkable since our analysis holds
for general nonexpansive maps and does not exploit the existence of 
 a potential nor the special properties of the Euclidean norm.
 On the other hand, our estimates are slightly more general since \cite{bm} 
 requires stronger assumptions: 
 $\nabla_x F(\cdot,\xi)$ is almost surely Lipschitz with a uniform constant $L$,
 and
the conditional variances $\EE(\|U_n\|_2^2|{\mathcal F}_{n-1})\leq \Var^2$
 are uniformly bounded almost surely. In contrast,
 we only require Lipschitzianity of the expected gradient $\nabla f(\cdot)$ and the weaker 
 condition $\EE(\|U_n\|_2^2)\leq \Var^2$. 
  
The best rates in \cref{eq:gradnorm,eq:gradnorm2} are obtained
for $a=2/3$, with orders  $O(1/\sqrt[6]{n})$ and $O(\log(n)/\sqrt[6]{n})$ respectively. A faster rate $O(1/\sqrt[4]{n})$ 
was obtained in \cite[Ghadimi and Lan]{GL2013} for a randomized version of 
\eqref{Eq:sgd2} which, instead of the last iterate $x_n$, outputs 
a random iterate $\hat x\in\{x_0,\ldots,x_n\}$ selected with
suitable probabilities. 
In the next section \S\ref{Sec:hilbert} we show that, 
for general nonexpansive maps in Euclidean spaces, 
the same improvement holds for a randomized version of \eqref{Eq:skm}. 

The asymptotic rates in gradient norm for stochastic optimization have been  improved  significantly
using more elaborated 
iterations which, however, do not seem to fit into the basic settings of \eqref{Eq:sgd2} nor \eqref{Eq:skm}. 
Up to our knowledge, the 
fastest rate is achieved by the {\sc (sgd{\small 3})} iteration in \cite[Allen-Zhu]{AZ18}
which attains a rate $O(1/\sqrt{n})$.

\section{Some comments on the Euclidean case}\label{Sec:hilbert}

Let us consider the special case where $\|\cdot\|$ is a norm in $\RR^d$ induced by an inner product  $\langle \cdot\,, \cdot \rangle$. 
Denote $R=d(x_0, \fix (T))$ the distance from $x_0$ to $\fix (T)$, and let $\mu$ as in \eqref{H2p} relating 
$\|\cdot\|$ to the standard $L_2$-norm.
By adapting the classical analysis of Browder and Petryshyn \cite{bp66} for the  deterministic Krasnoselskii--Mann iteration, {\em i.e.} when $U_n\equiv 0$ in \eqref{Eq:skm}, we have the following.

\begin{proposition} \label{Prop:HilbertBound}
Consider the iteration \eqref{Eq:skm} in a Euclidean space, and assume \eqref{H1}.  Then, for all $n\in\NN$ we have
\begin{equation}\label{Eq:HilbertBound}
\sum_{k=1}^n \alpha_k (1-\alpha_k)\, \EE \left ( \| Tx_{k-1} -x_{k-1}\|^2\right ) \leq R^2+ \mu^2\sum_{k=1}^{n} \alpha_{k}^2\Var_k^2.
\end{equation}
Equivalently, letting $p^{n}_k=\alpha_k (1-\alpha_k)/\sum_{i=1}^{n}\alpha_i (1-\alpha_i)$ for $k=1,\ldots,n$, and 
selecting a random iterate $\hat x_n\in\{x_0,\ldots,x_{n-1}\}$ with probabilities $\PP(\hat x_n=x_{k-1})=p^{n}_{k}$, we have
\begin{equation}\label{Eq:HilbertBound2}
\EE \left ( \|T\hat x_n -\hat x_n \|^2\right ) \leq \dfrac{R^2+ \mu^2\sum_{k=1}^{n} \alpha_{k}^2\Var_k^2}{\sum_{k=1}^{n} \alpha_k (1-\alpha_k)}.
\end{equation}
\end{proposition}

\begin{proof} 
Let $\zeta_k=\|x_k-y_0\|$  with $y_0$ the projection of $x_0$ onto $\fix (T)$. 
Using the parallelogram rule we obtain
\begin{eqnarray*}
\zeta_k^2 &=&\| (1\!-\!\alpha_k)x_{k-1}+ \alpha_k (Tx_{k-1} +U_k) -y_0\|^2\\
&=&  (1\!-\!\alpha_k)\zeta_{k-1}^2+ \alpha_k \|Tx_{k-1} - Ty_0 + U_k\|^2- \!\alpha_k(1\!-\!\alpha_k) \|Tx_{k-1} - x_{k-1}+U_k\|^2.
\end{eqnarray*}
Since  $Tx_{k-1} $ is $\mathcal F_{k-1}$-measurable and $\EE(U_k|\mathcal F_{k-1})=0$,
 using the law of total expectation and the nonexpansivity of $T$, it follows that
 \begin{equation*}
\begin{aligned}
 \EE\left(\|Tx_{k-1} - Ty_0 + U_k\|^2\right)~&=~ \EE\left(\|Tx_{k-1} - Ty_0\|^2 + 2 \langle Tx_{k-1} - Ty_0, U_k \rangle + \|U_k\|^2\right)\\
 &\leq~\EE \left ( \zeta_{k-1}^2\right ) +\EE\left(\|U_k\|^2\right),
 \end{aligned}
\end{equation*}
and similarly
\[
\EE \left (  \|Tx_{k-1} \!-\! x_{k-1} \!+\! U_k\|^2 \right )\!=\!  \EE \left ( \|Tx_{k-1}\! -\! x_{k-1}\|^2 \right ) \!+\! \EE\left(\|U_k\|^2\right).
\]
Combining these inequalities, we deduce
\[
\EE \left ( \zeta_k^2\right )  \leq 
\EE \left (  \zeta_{k-1}^2 \right ) +  \alpha_k^2\,\EE\left(\|U_k\|^2\right) -  \alpha_k(1-\alpha_k)\,\EE \left ( \|Tx_{k-1} - x_{k-1}\|^2 \right ),
\]
and then \eqref{Eq:HilbertBound} follows by telescoping and using the bound $\EE(\|U_k\|^2)\leq\mu^2\Var_k^2$. 

\end{proof}
\begin{remark} 
Note that  \eqref{Eq:HilbertBound2} implies a bound for the minimum expected residual
\[
\min_{k=0,\ldots, n-1} \EE \left ( \|Tx_k -x_k \|^2\right ) \leq \dfrac{R^2+ \mu^2\sum_{k=1}^{n} \alpha_{k}^2\Var_k^2}{\sum_{k=1}^{n} \alpha_k (1-\alpha_k)},
\]
however detecting the $k$ at which the minimum is attained would require to compute the expected values, which is not possible in general. In contrast, the random output $\hat x_n$ is easily  implemented.
\end{remark}

\begin{remark} 
In the deterministic case we have that $\norm{x_n -Tx_n}$ is nonincreasing, and one can obtain a bound for the last-iterate residual, namely
\[
  \|Tx_n -x_n \|^2 \leq \frac{R^2}{\sum_{k=1}^{n+1} \alpha_k (1-\alpha_k)}.
\]
However, in the stochastic case it seems unlikely to have $\EE \left ( \| x_n - Tx_n\|^2 \right )$
decreasing so  the error bound
\eqref{Eq:HilbertBound2} is stated for the random iterate $\hat x_n$. 
\end{remark}

In order to compare these  Euclidean space estimates with our previous error bounds for the expected fixed-point residual (instead of its square), we can apply Jensen's inequality to \eqref{Eq:HilbertBound2} to obtain
\begin{equation}\label{Eq:HilbertBound3}
\EE \left ( \|T\hat x_n -\hat x_n \|\right ) \leq \sqrt{\dfrac{R^2+ \mu^2\sum_{k=1}^{n} \alpha_{k}^2\Var_k^2}{\sum_{k=1}^{n} \alpha_k (1-\alpha_k)}}.
\end{equation}
We conclude by analyzing the case of bounded variances with
constant stepsizes and stepsizes of power form. 
\begin{corollary}\label{CorHilbertBnd0}
 Assume  \eqref{H1} and \eqref{Eq:BoundedVar}.
Consider \eqref{Eq:skm} in a  Euclidean space over a fixed horizon $n_0$ with constant stepsizes $\alpha_n\equiv \alpha={1}/{\sqrt{n_0+1}}$.
Let $\hat x_{n_0}$ be a random iterate  as in  \cref{Prop:HilbertBound}. Then,
 \begin{equation}\label{EQ:HilbertBnd}
\EE \left ( \|T\hat x_{n_0} -\hat x_{n_0} \|\right ) \leq  \dfrac{\sqrt{R^2+\mu^2\Var^2}}{\sqrt{\tau_1(1/2)}\;\sqrt[4]{n_0}}.
\end{equation}

\end{corollary}

As in \cref{Thm:Rate_Expectation_constant} one can find the optimal constant stepsize $\alpha$ for a fixed horizon $n_0$, 
which  yields a similar error bound with a slightly smaller constant.
Note also that the rate $O(1/\sqrt[4]{n_0})$ in \eqref{EQ:HilbertBnd} is faster than
the $O(1/\sqrt[6]{n_0})$ from 
 \cref{Thm:Rate_Expectation_constant}.
However, these results are not directly comparable. Indeed,  \cref{Thm:Rate_Expectation_constant} 
 holds for the last 
iterate $x_{n_0}$ and applies to general normed spaces but under \eqref{Eq:BoundedKappa}, whereas 
\cref{CorHilbertBnd0} is valid for the random iterate $\hat x_{n_0}$ and 
in  Euclidean spaces although it only requires $\fix (T)$ to be nonempty.

For stepsizes of power form we get the next  estimate, with $\tau_n(a)$ as in \eqref{eq:taua}.

\begin{corollary}\label{corEq:HBP}
Assume  \eqref{H1} and \eqref{Eq:BoundedVar}.
Let us consider the  iteration \eqref{Eq:skm} in a Euclidean space, for $\alpha_n={1}/{(n+1)^a}$ with $\frac{1}{2}\leq a\leq 1$.
Let $\hat x_n$ be a random iterate as in  \cref{Prop:HilbertBound}. 
Then,
\begin{equation}\label{Eq:HBP}
\EE \left ( \|T\hat x_n -\hat x_n \|\right ) \leq
\left\{\begin{array}{ll} 
\dfrac{\sqrt{R^2+ \mu^2\Var^2\,\ln{(n+1)}}}{\sqrt{\tau_1(1/2)}\;\sqrt[4]{n}} & \text{if }~ a=\frac{1}{2},\\[3ex]
\dfrac{\sqrt{R^2+\mu^2\Var^2/(2a-1)}}{\sqrt{\tau_1(a)}\;n^{(1-a)/2}}& \text{if }~ \frac{1}{2}<a<1,\\[3ex]
\dfrac{2\sqrt{\ln 2}\;\sqrt{R^2+ \mu^2\Var^2}}{\sqrt{\ln{(n+1)}}} & \text{if }~ a=1.
\end{array}
\right.
\end{equation}
\end{corollary}

\begin{proof} 
This follows from \eqref{Eq:HilbertBound3} by combining the simple integral upper estimate 
$$\sum_{k=1}^{n} \alpha_k^2=\sum_{k=1}^{n} \frac{1}{(k+1)^{2a}}\leq\int_1^{n+1}\frac{1}{x^{2a}}dx\leq\left\{\begin{array}{cl}
{1}/{(2a-1)}&\mbox{\rm if } \frac{1}{2}<\alpha\leq 1,\\
\ln (n+1)&\mbox{\rm if } \alpha=\frac{1}{2},
\end{array}\right.$$
with the lower estimates $\tau_n(a)\geq\tau_1(a)\; n^{1-a}$ for $\frac{1}{2}\leq a<1$, and $\tau_n(1)\geq \frac{1}{4\ln 2}\,\ln (n+1)$ for $a=1$, established in  \cref{Lemma:L41} in the \cref{appendix1}.
\end{proof}

Note that for $a=\frac{2}{3}$ the bound 
\eqref{Eq:HBP} is of order $O(1/\sqrt[6]{n})$, which is slightly faster than the rate $O(\ln(n)/\sqrt[6]{n})$ in  \cref{Thm:RatePower}, while for $\frac{2}{3}<a\leq 1$ both analysis yield the same rate $O(1/n^{(1-a)/2})$.
On the other hand,
while the rates in  \cref{Thm:RatePower}
deteriorate for $a<\frac{2}{3}$, in the Euclidean setting of 
 \cref{corEq:HBP} the rates improve
attaining an order
$O(\ln(n)/\sqrt[4]{n})$ for $a=\frac{1}{2}$.
 However, as mentioned before, these bounds apply in different settings: the random iterate $\hat x_n$
and Euclidean spaces for \eqref{Eq:HBP}, and the last iterate
$x_n$ and general normed spaces in  \cref{Thm:RatePower}. 
  \newpage
\section*{Acknowledgments}
We are deeply indebted to 4 reviewers for their 
insightful comments on the original manuscript which prompted us to enhance 
the scope of the paper, notably by including the  application to RVI-$Q$-learning. 
We also warmly thank our friend and colleague Crist\'obal Guzm\'an for stimulating discussions on this topic and, in particular,  
for pointing out the interpretation of \eqref{Eq:HilbertBound} 
as an expected value for a random selection of the iterates,
and the resulting rate $O(1/\sqrt[4]{n_0})$ for constant stepsizes.

\bibliographystyle{ims}
\bibliography{ref_bc}

 \appendix

\section{ Proof of  \texorpdfstring{  \cref{Thm:RatePower}}{Lg}}\label{appendix00}

In this Appendix we establish explicit error bounds in the case of  stepsizes of power form $\alpha_n=1/(n+1)^a$ with $\frac{1}{2}\leq a\leq 1$
and bounded variances $\Var_n\leq\Var$.
We adopt the notations in  \eqref{eq:power}, and for brevity we just write $\tau_n$ for the expression $\tau_n(a)$ given by \cref{eq:taua}.
The analysis is split into three Lemmas that deal with different ranges for the exponent $a$,  and which are based on some additional technical results presented in the next \cref{appendix1}. These error bounds are then used to prove \cref{Thm:RatePower}.
 
We first deal with the case $a<1$. In view of \cref{Lemma:kappa} we can use the estimate  \eqref{Eq:cotaalpha} and
apply the integral bound \eqref{Eq:RateExpected2} with 
 $$h(x)=\Var\,\mbox{$(\lambda_a+(1-a) x)^{-\frac{a}{2 (1-a)}}.$}$$ 
 The corresponding integral in $B_n$ can be expressed in closed form in terms of Gauss' hypergeometric function 
 $_2F_1(a,b;c;z)$,
 which yields the following error bound. 
 
\begin{proposition}\label{Cor:ErrorBoundPower} Assume \eqref{H1} and \eqref{Eq:BoundedKappa}.
Let $\alpha_n={1}/{(n\!+\!1)^a}$ with $\frac{1}{2}\leq a<1$, and denote $b=\frac{a}{2 (1-a)}$ and $z_{n}= \frac{(1-a)(\tau _n-\tau _1)}{\lambda_a+(1-a) \tau _n}$. Then
 $\EE(\|x_n-Tx_n\|)\leq\! B_{n,a}$ with
\begin{equation*}
B_{n,a}\triangleq \mbox{ \large
$\frac{\bar\kappa}{\sqrt{\pi\tau_n}}+
\frac{4\normeq \Var}{\sqrt{\pi}}\frac{\gamma_a\sqrt{\tau_n -\tau_1}}{ (\lambda_a +(1-a)\tau_n)^{b}} $}\, \mbox{\normalsize $_2F_1\!\left(b,\frac{1}{2};\frac{3}{2};z_{n}\right)$}+\mbox{\small $2 \normeq \Var \,(\alpha_n\!+\!2)\sqrt{\alpha_{n}}$}.
\end{equation*}
\end{proposition}
\begin{proof} This follows by majorizing the two last terms in \eqref{Eq:RateExpected2} 
by using  \cref{Lemma:L3,Lemma:L1} which give respectively  $\sumprob_{n-1}\leq \Var\sqrt{\alpha_{n}}$ and $\sumprob_{n}\leq\Var\sqrt{\alpha_{n}}$,
and computing the integral term in the bound using the change of variables $y=\frac{\tau_n-x}{\tau_n-\tau_1}$ which yields
\begin{eqnarray}\label{Eq:IntegralBound_n^a}
 \int_{\tau_1}^{\tau_n}\!\!\frac{h(x)}{\sqrt{\tau_n-x}}\,dx&=&
 \mbox{$
 \frac{\Var\,\sqrt{\tau _n-\tau _1}}{{(\lambda_a+(1-a) \tau _n )^{b}}}$}\int _0^1\mbox{$\frac{1}{\sqrt{y}}(1-z_{n}\,y)^{-b}dy$}\\\nonumber
 &=& \mbox{$\frac{2\Var\,\sqrt{\tau_n -\tau_1}}{ (\lambda_a +(1-a)\tau_n)^{b}} \, _2F_1\!\left(b,\frac{1}{2};\frac{3}{2};z_{n}\right)$}
\end{eqnarray}
where the last equality results by matching the coefficients in the known identity
$$_2F_1(a,b;c;z)=\mbox{$\frac{\Gamma(c)}{\Gamma(b)\Gamma(c-b)}$}\int_0^1y^{b-1}(1-y)^{c-b-1}(1-z \,y)^{-a} dy.$$

\vspace{-4ex}
 
\ \end{proof}

\vspace{2ex}
For $a=\frac{2}{3}$ we get the following simpler expression for the previous error bound. 
\begin{proposition}\label{Cor:ErrorBoundPower1} Assume \eqref{H1} and \eqref{Eq:BoundedKappa}, and  let $c=3\,\lambda_{2/3}\approx 1.76393$ and $\gamma_{2/3}=\frac{\alpha_4(1-\alpha_4)}{\alpha_{5}(1-\alpha_{5})}\approx 1.06584$. Then
\begin{equation}\label{Eq:RateExpected4}
B_{n,\frac{2}{3}}=\!\!\mbox{ \large
$\frac{\bar\kappa}{\sqrt{\pi\tau_n}}\!+\!
 \frac{6\normeq \Var\gamma_{2/3}}{\sqrt{\pi} \sqrt{c +\tau_n }} $}\, \mbox{\small $\ln \left(1+2\frac{(\tau_n -\tau _1)+\sqrt{ (c +\tau_n)(\tau_n -\tau _1)}}{c +\tau_1 }\right)$}\!+\!\mbox{\small $2 \normeq \Var \,(\alpha_n\!+\!2)\sqrt{\alpha_{n}}$}.
\end{equation}
\end{proposition}
\begin{proof} This follows from the identity $_2F_1(1,\frac{1}{2};\frac{3}{2};z)=\frac{1}{2\sqrt{z}}\ln(\frac{1+\sqrt{z}}{1-\sqrt{z}})$.
\end{proof}

 \vspace{2ex} 
 
 Let us consider next the case $a=1$, that is to say, $\alpha_n={1}/{(n+1)}$.
In this case $\tau_n=\sum_{k=1}^n\frac{k}{(k+1)^2}\approx\ln n$, and in fact  $\tau_n-\ln (n\!+\!1)\in (\gamma-\frac{\pi^2}{6},0)$
with $\gamma\approx 0.5772$  the Euler-Mascheroni constant.
In view of   \cref{Lemma:alfa1}, we can apply \eqref{Eq:RateExpected2}  by 
using $h(x)=\Var\,\omega\, e^{-x/2}$ with $\omega=\frac{\sqrt{3}}{2}\,e^{17/72}$. The resulting
convolution integral in the error bound can be expressed in closed form in terms of Dawson's function
$$D_+(x)\triangleq e^{-x^2}\int_0^xe^{y^2}dy\;\approx\;\frac{1}{2x}$$
where the asymptotics for  $x$ large follows from l'H\^opital's rule.
\begin{proposition}\label{le:B3}Assume  \eqref{H1} and \eqref{Eq:BoundedKappa}. Let 
$\alpha_n=1/{(n+1)}$ and denote $d_1=\sqrt{3}\,e^{1/9}\,\frac{32}{27}$. Then, for all $n\geq 1$ we have 
$\EE(\|x_n-Tx_n\|)\leq\! B_{n,1}$ with 
\begin{equation}\label{Eq:RateExpected5}
B_{n,1}\triangleq
\frac{\bar\kappa}{\sqrt{\pi\tau_n}}+
\frac{2\sqrt{2}\,\normeq \Var d_1}{\sqrt{\pi}}\,D_+\!\left(\sqrt{\mbox{$\frac{1}{2}$}\left(\tau_n\!-\!\mbox{$\frac{1}{4}$}\right)}\right)+\frac{2 \normeq \Var (2n+3)}{(n+1)^{3/2}}
\end{equation}
and $B_{n,1}\approx\;\frac{\bar\kappa+2\mu\Var d_1}{\sqrt{\pi}}\frac{1}{\sqrt{\ln n}}$ as $n\to\infty$.
\end{proposition} 
\begin{proof} Considering $\tau_1=\frac{1}{4}$ and using the change of variables $y=\sqrt{(\tau \!-\!x)/2}$,
the convolution integral in the error bound  \eqref{Eq:RateExpected2}  can be computed as
\begin{eqnarray*}
\int_{\tau_1}^{\tau}\frac{\omega\,e^{-x/2}}{\sqrt{\tau-x}}\,dx
&=&\mbox{$2 \sqrt{2} \,\omega $} e^{-\frac{\tau}{2}}\int_0^{\sqrt{(\tau-\tau_1)/2}} e^{y^2} \, dy\\
&=&\mbox{$\sqrt{6}\,e^{1/9}\,D_+\!\big(\mbox{\scriptsize $\sqrt{\left(\tau\!-\!\frac{1}{4}\right)/2}$}\big)$}\\
&\approx&\mbox{$\frac{\sqrt{3}\,e^{1/9}}{\sqrt{\tau}}$}.
\end{eqnarray*}
Then \eqref{Eq:RateExpected5} follows by substituting in \eqref{Eq:RateExpected2} with $\tau=\tau_n\approx\ln n$,
and noting that $\frac{\alpha_k(1-\alpha_k)}{\alpha_{k+1}(1-\alpha_{k+1})}$
is decreasing so that its maximum $\gamma_1=\frac{32}{27}$ is attained for $k=2$.
\end{proof}

\vspace{1ex}
Combining the estimates in these preliminary Lemmas, we now proceed to present the proof of  \cref{Thm:RatePower}.

\begin{proof}[\rm\bf Proof of  \cref{Thm:RatePower}] 
For $a=1$ the tight asymptotics of $B_{n,1}$ in \eqref{Eq:ratefinal}
was already established in \cref{le:B3}. 
Let us  then consider the case $\frac{1}{2}\leq a<1$.

We estimate the asymptotic rate for the three terms in the 
 bounds $B_{n,a}$ in \cref{Cor:ErrorBoundPower}. The simplest one is the third term $I_3\triangleq 2 \normeq \Var \,(\alpha_n\!+\!2)\sqrt{\alpha_{n}}\approx 4\normeq \Var \, n^{-a/2}$, 
 which holds for all 
$\frac{1}{2}\leq a< 1$.
On the other hand, an integral estimate for the sum $\tau_n$ yields $\tau_n\approx\frac{n^{1-a}}{1-a}$,
so that the first term in $B_{n,a}$ is $I_1\triangleq \frac{\bar\kappa}{\sqrt{\pi\tau_n}}\approx\frac{\bar\kappa\sqrt{1-a}}{\sqrt{\pi}}\, n^{-(1-a)/2}$.
Note that for $a>\frac{1}{2}$ and $n$ large $I_3$ is negligible compared to $I_1$, whereas for $a=\frac{1}{2}$ they
are both of the same order $O(n^{-1/4})$.

Let us now analize the rate of  the second term $I_2\triangleq \frac{2\normeq \gamma}{\sqrt{\pi}}\! \int_{\tau_1}^{\tau_n}\!\!\!\frac{h(x)}{\sqrt{\tau_n\!-\!x}}\,dx$,
which we separate in three cases depending on the location of $a$.

\vspace{1ex}
\noindent\underline{\sc Case $a\in[\frac{1}{2},\frac{2}{3})$}. From \eqref{Eq:IntegralBound_n^a} with $b=\frac{a}{2(1-a)}$ and $z_{n}= \frac{(1-a)(\tau _n-\tau _1)}{\lambda_a+(1-a) \tau _n}$, we have
\begin{eqnarray}\label{Eq:intnosing}
 I_2\triangleq \frac{2\normeq\gamma_a}{\sqrt{\pi}}\! \int_{\tau_1}^{\tau_n}\!\!\!\frac{h(x)}{\sqrt{\tau_n\!-\!x}}\,dx
&=&\mbox{$ \frac{2\normeq \Var\gamma_a}{\sqrt{\pi}} \frac{\sqrt{\tau _n-\tau _1}}{{(\lambda_a+(1-a) \tau _n )^{b}}}$}\int _0^1\mbox{$\frac{1}{\sqrt{y}}(1-z_{n}\,y)^{-b}dy$}\\\nonumber
&\approx&\mbox{$ \frac{2\normeq \Var\gamma_a}{\sqrt{\pi(1-a)}} n^{\frac{1}{2}-a}$}\int _0^1\mbox{$\frac{1}{\sqrt{y}}(1-z_{n}\,y)^{-b}dy$}
\end{eqnarray}
We observe that $z_n$ increases and converges to $1$ and, since $b<1$, by monotone convergence 
the latter integral has a finite limit when $n\to\infty$, namely\footnote{Recall the definition of the gamma function $\Gamma(z)=\int_0^\infty x^z e^{-x}dx$}
$$\int _0^1\mbox{$\frac{1}{\sqrt{y}}(1-y)^{-b}dy=\sqrt{\pi} \,\frac{\Gamma(1-b)}{\Gamma(\frac{3}{2}-b)}.$}$$
Hence $I_2\approx \frac{2\normeq \Var\gamma_a}{\sqrt{1-a}}\frac{\Gamma(1-b)}{\Gamma(\frac{3}{2}-b)} n^{\frac{1}{2}-a}$ and
since $\frac{1}{2}-a < \frac{1}{2}(1-a)$ it follows that $I_2$ dominates $I_1$ and $I_3$ in the error bound. This
proves \eqref{Eq:ratefinal} for $a\in[\frac{1}{2},\frac{2}{3})$.

\vspace{1ex}
\noindent\underline{\sc Case $a=\frac{2}{3}$}. Here we have $\tau_n\approx 3 \,n^{1/3}$ and one can readily check that the dominating term is precisely $I_2$
with $I_2 \approx \frac{2\normeq \Var\gamma_{2/3}}{\sqrt{3\pi}} \frac{\ln n}{n^{1/6} }$ which follows directly from \eqref{Eq:RateExpected4}.
This establishes the case $a=\frac{2}{3}$ in \eqref{Eq:ratefinal}.

\vspace{1ex}
\noindent\underline{\sc Case $a\in(\frac{2}{3},1)$}. Here we have $b>1$ and since $z_n\to 1$ the integral  on the left-hand side of \eqref{Eq:intnosing} diverges as $n\to\infty$.
To deal with this case we use the original expression which we write as
$$
 I_2=\frac{2\normeq \Var\gamma_a}{\sqrt{\pi\tau_n}}\!  \int_{0}^{\infty}\!\frac{(\lambda_a+(1-a) x)^{-b}}{\sqrt{1\!-\!x/\tau_n}}\mathbbm{1}_{[\tau_1,\tau_n]}(x)\,dx.$$
 Using Lebesgue's dominated convergence theorem, the latter integral converges to 
 $$\int_{\tau_1}^{\infty}\!(\lambda_a+(1-a) x)^{-b}\,dx=\mbox{$\frac{2(\lambda_a+(1-a) \tau _1 )^{1-b}}{3a-2}$.}$$
It follows that $I_2\approx\frac{2\normeq \Var d_a}{\sqrt{\pi \tau_n}}$ which is of the same order as $I_1\approx\frac{\bar\kappa}{\sqrt{\pi\tau_n}}$, and both terms combined
yield the asysmptotics \eqref{Eq:ratefinal} for $a\in(\frac{2}{3},1)$.
 \end{proof}
 
\section{Estimates for stepsizes of power form}\label{appendix1}

This Appendix considers again the case of stepsizes of power form $\alpha_n=\frac{1}{(n+1)^a}$ with $0<a\leq 1$,
and establishes the technical estimates required to prove the explicit error bounds in {\color{blue} \cref{appendix00}}.
As before we adopt the notations in  \eqref{eq:power}, and we write $\tau_n$ for the expression $\tau_n(a)$ in \cref{eq:taua}.

\subsection{Bounds for   \texorpdfstring{$\delta_n^2=\sum_{k=1}^n(\pi_k^n)^2$}{Lg}}
 \begin{lemma}  \label{Lemma:L1}
 For $\alpha_n=\frac{1}{(n+1)^a}$ with $0<a\leq 1$  we have $\frac{1}{2}\alpha_n\leq \delta_n^2\leq \alpha_n$ for $n\geq 1$.
  \end{lemma}
\begin{proof} For the upper bound we claim that $\pi_k^n$ increases for $k\geq 1$ so that $\pi_k^n\leq\pi_n^n=\alpha_n$,
and then  $$\mbox{$\delta_n^2=\sum_{k=1}^n(\pi_k^n)^2\leq \pi_n^n\sum_{k=1}^n\pi_k^n= \pi_n^n(1-\pi_0^n)\leq\pi_n^n=\alpha_n$}.$$
To prove the monotonicity we note that $\pi_{k}^n\leq\pi_{k+1}^n$ amounts to
$ \alpha_{k}(1\!-\!\alpha_{k+1})\leq \alpha_{k+1}$, which can be expressed as $(k+1)^a-1\leq k^a$.
Dividing by $(k+1)^a$ this becomes
$1\leq (\frac{1}{k+1})^a+(1-\frac{1}{k+1})^a$. The latter inequality holds because the map $x\mapsto x^a+(1-x)^a$ 
 is concave for $x\in[0,1]$ with minimum value 1 attained at the border of the interval. 
 
Let us now establish the lower bound $ \delta_n^2\geq\frac{1}{2}\alpha_n$. For $n=1$ this holds since
$\delta_1^2=\alpha_1^2$ and $\alpha_1=\frac{1}{2^a}\geq\frac{1}{2}$. For the inductive step, if  $\delta_{n-1}^2\geq \frac{1}{2}\alpha_{n-1}$ we note
that $$\delta_n^2=(1-\alpha_n)^2\delta_{n-1}^2+\alpha_n^2\geq\mbox{$ \frac{1}{2}$}(1-\alpha_n)^2\alpha_{n-1}+\alpha_n^2$$
so that it suffices to show that $(1-\alpha_n)^2\alpha_{n-1}+2\alpha_n^2\geq \alpha_n$, or equivalently
$$\mbox{$(1-\alpha_n)^2\geq (1-2\alpha_n)\frac{\alpha_{n}}{\alpha_{n-1}}$}.$$
Since $\frac{\alpha_{n}}{\alpha_{n-1}}=(\frac{n}{n+1})^a=(1-\alpha_n^{1/a})^{a}$,
it suffices to show  $(1-\alpha)^2\geq (1-2\alpha)(1-\alpha^{1/a})^a$ for all $\alpha\in[0,1]$. When $\alpha\geq\frac{1}{2}$ this is clear since
the right-hand side becomes negative. For $\alpha\leq\frac{1}{2}$ the inequality follows since  $(1-\alpha)^2\geq (1-2\alpha)$ and $(1-\alpha^{1/a})^a\leq 1$.
\end{proof}

Below we establish a slightly sharper upper bound.
The proof  is elementary but somewhat  technical so we split it into several Lemmas.

\begin{lemma}\label{Lemma:L3} For $\alpha_n=\frac{1}{(n+1)^a}$ with $0<a\leq 1$ we have $\delta_n^2\leq\alpha_{n+1}$ for all $n\geq 1$.
\end{lemma}
\begin{proof}  Let us set $\rho_n\triangleq \delta_n^2/\alpha_{n+1}$ and prove by induction that $\rho_n\leq 1$.
For $n=1$ this holds since $\rho_1=\alpha_1^2/\alpha_2=(\frac{3}{4})^a\leq 1$. 
For the inductive step, we take $n\geq 2$ such that $\rho_{n-1}\leq 1$. We will exploit two alternative
 upper bounds. As in the proof of  \cref{Lemma:L1}  we have $\delta_n^2\leq\pi^n_n(1-\pi_0^n)$ and therefore
\begin{equation}\label{Eq:Cota2}
\rho_n\leq \mbox{$\frac{\alpha_n}{\alpha_{n+1}}(1-\prod_{i=1}^n(1-\alpha_i)).$}
\end{equation}
 The second bound is derived from the recursive formula $\delta_n^2=(1-\alpha_n)^2\delta_{n-1}^2+\alpha_n^2$,
 which combined with the induction hypothesis yields
\begin{equation}\label{Eq:Cota1}
\rho_n\leq\mbox{$ \frac{\alpha_n}{\alpha_{n+1}}(1-\alpha_n+\alpha_n^2)=  \frac{\alpha_n}{\alpha_{n+1}}(1-\alpha_n(1-\alpha_n))$}.
\end{equation}

In the sequel we consider $\alpha_n=(n+1)^{-a}$ as a function of $a$, and we analyze 
how the previous  bounds vary with $a$.
We first note that these bounds coincide if and only if $\alpha_n=\prod_{i=1}^{n-1}(1-\alpha_i)$, 
that is to say, $\frac{1}{(n+1)^a}=\prod_{i=1}^{n-1}(1-\frac{1}{(i+1)^a})$. 
When $a$ increases from 0 to 1 the left-hand side decreases from 1 to $\frac{1}{n+1}$
whereas the expression on the right increases from 0  to $\frac{1}{n}$. Thus, there is a 
unique $a_n\in(0,1)$ at which both bounds coincide: the bound \eqref{Eq:Cota1} is smaller
than \eqref{Eq:Cota2} for $a\in[0,a_n]$, and conversely for  $a\in[a_n,1]$.

We will show that both bounds are smaller than 1 in the corresponding intervals, or equivalently,
that their logarithms are smaller than 0. 
To this end we exploit the following properties (established later in separate Lemmas)

\begin{itemize}
\item[a)] $u_n(a)\triangleq a\ln(\frac{n+2}{n+1})+\ln(1-\prod_{i=1}^n(1-\frac{1}{(i+1)^{a}}))$ is concave,
\item[b)]  $v_n(a)\triangleq a\ln(\frac{n+2}{n+1})+\ln(1-\frac{1}{(n+1)^{a}}+\frac{1}{(n+1)^{2a}})$ first decreases and then increases.
\end{itemize}
Since  clearly $u_n(0)=v_n(0)=0$, the conclusion follows directly if we show that the common 
value $u_n(a_n)=v_n(a_n)$ at the intersect is negative. Linearizing the concave function $u_n(\cdot)$ it suffices to show that 
\begin{equation}\label{Eq:un_negativo}
 u_n(1) + u_n'(1)(a_n-1)\leq 0.
 \end{equation}
A direct computation gives $u_n(1)=\ln(\frac{(n+2)n}{(n+1)^2})$ and
$u_n'(1)= \ln(\frac{n+2}{n+1})-\frac{1}{n}\sum_{i=1}^n\frac{\ln(i+1)}{i}$, while
an inductive argument shows that $u_n'(1)$ decreases in $n$ with $u_2(1)<0$ so that $u_n'(1)< 0$.
On the other hand, we observe that $a_n\geq 1-b_n$ with
$$b_n=\mbox{$\frac{\ln(n+1)-\ln(n)}{\ln(n+1)+\sum_{i=1}^{n-1}\frac{\ln{i+1}}{i}}$}$$
which follows by taking logarithms in the equation that defines $a_n$, and using the next inequality obtained
by linearizing the concave function $w_n(a)=\sum_{i=1}^{n-1}\ln(1-\frac{1}{(i+1)^a})$
$$-a_n\ln(n+1)=w_n(a_n)\leq w_n(1)+w_n'(1)(a_n-1)=\mbox{$-\ln(n)+(\sum_{i=1}^{n-1}\frac{\ln{i+1}}{i})(a_n-1)$}.$$
Hence, \eqref{Eq:un_negativo} will follow by showing that $u_n(1)- u_n'(1)b_n \leq 0$, which is done in  \cref{Lemma:Final} below.
\end{proof}

The next two Lemmas establish respectively the claims a) and b) used in the previous proof.
\begin{lemma}
$u_n(a)\triangleq a\ln(\frac{n+2}{n+1})+\ln(1-\prod_{i=1}^n(1-\frac{1}{(i+1)^{a}}))$ is concave for all $n\geq 1$.
\end{lemma}
\begin{proof}
For $n=1$ this holds trivially since $u_1(a)= a\ln(\frac{3}{4})$, so let us consider the case $n\geq 2$.
Ignoring the linear term $a\ln(\frac{n+2}{n+1})$, it suffices to show the concavity of
$w_n(a)\triangleq \ln(1-g_n(a))$ where $g_n(a)\triangleq \prod_{i=1}^n(1-(i+1)^{-a})$.
Let us first observe that 
$$(\ln g_n)'(a)=\mbox{$\sum_{i=1}^n\frac{\ln(i+1)}{(i+1)^{a}-1}\triangleq S_1(a)\geq 0$}$$
so that $g_n'=g_n\,S_1$. Setting $S_2(a)=-S_1'(a)=\mbox{$\sum_{i=1}^n\frac{(i+1)^{a}\ln^2(i+1)}{((i+1)^{a}-1)^2}$}$ 
we also get
$$g_n''=g_n'\,S_1+g_n\,S_1'=g_n\,(S_1^2-S_2).$$
Using these identities  and $w_n'=-g_n'/(1-g_n)$ we obtain
\begin{eqnarray*}
w_n''&=&-\mbox{$\frac{1}{(1-g_n)^2}[g_n''(1-g_n)+(g_n')^2]$}\\
&=&-\mbox{$\frac{g_n}{(1-g_n)^2}[S_1^2-(1-g_n)S_2].$}\\
&=&-\mbox{$\frac{g_n}{(1-g_n)^2}\sum_{i=1}^n\frac{\ln(i+1)}{(i+1)^a-1}[S_1-(1-g_n)\frac{\ln(i+1)}{1-(i+1)^{-a}}].$}
\end{eqnarray*}
We claim  that all the terms $S_1-(1-g_n)\frac{\ln(i+1)}{1-(i+1)^{-a}}$ in the latter sum 
are non-negative. As a matter of fact, since the expression $\frac{\ln(i+1)}{1-(i+1)^{-a}}$
increases with $i$, it suffices to show  
\begin{equation}\label{Eq:Dec1}
\mbox{$S_1\geq (1-g_n)\frac{\ln(n+1)}{1-(n+1)^{-a}}\qquad(\forall\,n\geq 2),$}
\end{equation}
for which we consider 3 separate cases.

\vspace{1ex}
\noindent\underline{\sc Case $n=2$}. In this case the inequality \eqref{Eq:Dec1} reads 
$$\mbox{$\frac{\ln(2)}{2^a-1} + \frac{\ln(3)}{3^a-1} $}\geq\mbox{$ \frac{(1-(1-2^{-a})(1- 3^{-a})) \ln(3)}{(1 - 3^{-a})}$}=\mbox{$(\frac{1}{3^a-1}+\frac{1}{2^{a}}) \ln(3)$}$$
which is equivalent to $2^{-a}\geq 1- \frac{\ln(2)}{\ln(3)}$ and clearly holds for all $a\in [0,1]$.

\vspace{1ex}
\noindent\underline{\sc Case $n=3$}. Here \eqref{Eq:Dec1} becomes
$$\mbox{$\frac{\ln (2)}{2^a-1}+\frac{\ln (3)}{3^a-1}+\frac{\ln (4)}{4^a-1}$}\geq\mbox{$ \frac{\left(1-(1-2^{-a})(1-3^{-a})(1-4^{-a})\right) \ln (4)}{(1-4^{-a})}$}=\mbox{$( \frac{1}{4^a-1}+\frac{1}{2^{a}}+\frac{1}{3^{a}}-\frac{1}{6^{a}}) \ln (4)$}
$$
which can be rewritten as
$$\mbox{$\frac{\ln (2)}{2^a-1}+\frac{\ln (3)}{3^a-1}$}\geq\mbox{$ (\frac{1}{2^{a}}-\frac{1}{2} \frac{1}{6^a})\ln (4)+(\frac{1}{3^{a}}-\frac{1}{2} \frac{1}{6^a}) \ln (4)$}.$$
This follows by combining the inequalities $\frac{\ln (2)}{2^a-1}\geq (\frac{1}{2^{a}}-\frac{1}{2} \frac{1}{6^a})\ln (4)$ and
$\frac{\ln (3)}{3^a-1}\geq (\frac{1}{3^{a}}-\frac{1}{2} \frac{1}{6^a}) \ln (4)$, both of which can be checked
by noting that they hold for $a=1$ and observing that  $(2^a-1)(\frac{1}{2^{a}}-\frac{1}{2} \frac{1}{6^a})$ and
$(3^a-1)(\frac{1}{3^{a}}-\frac{1}{2} \frac{1}{6^a})$ are increasing with $a$. 

\vspace{1ex}
\noindent\underline{\sc Case $n\geq 4$}. We will prove the stronger inequality $S_1\geq \frac{\ln(n+1)}{1-(n+1)^{-a}}$, which we rewrite
$$\mbox{$\sum _{i=1}^{n-1} \frac{\ln (i+1)}{(i+1)^a-1}\geq \frac{\ln (n+1)}{1-(n+1)^{-a}}-\frac{\ln (n+1)}{(n+1)^a-1}=\ln (n+1)$}.$$
Since $\sum _{i=1}^{n-1} \frac{\ln (i+1)}{(i+1)^a-1}$ decreases with $a$, it suffces to check at $a=1$,
that is to say
$$\mbox{$\sum _{i=1}^{n-1} \frac{\ln (i+1)}{i}\geq \ln (n+1),$}$$
which readily follows by induction by showing that $\Var_n\triangleq \sum _{i=1}^{n-1} \frac{\ln (i+1)}{i}-\ln (n+1)$ is 
increasing for $n\geq 4$ with $\Var_4>0$.
\end{proof}

\begin{lemma}$v_n(a)\triangleq a\ln(\frac{n+2}{n+1})+\ln(1-\frac{1}{(n+1)^{a}}+\frac{1}{(n+1)^{2a}})$ first decreases and then increases
\end{lemma}
\begin{proof}
A direct computation shows that $v_n'(a)\leq 0$ is equivalent to
$$\mbox{$\frac{2-(n+1)^a}{(n+1)^{2a}-(n+1)^a+1}\geq\frac{\ln(n+2)}{\ln(n+1)}-1.$}$$
Setting $c_n=1/(\frac{\ln(n+2)}{\ln(n+1)}-1)$ and $z=(n+1)^a$ this is transformed into
$$ z^2+(c_n-1)z+(1-2c_n)\leq 0$$
Since $c_n>1$ the latter holds if and only if $z\leq z_n$ with $z_n=\frac{1}{2}(1-c_n+\sqrt{c_n^2+6c_n-3})$ the positive root of the quadratic.
Therefore, $v_n(a)$ decreases for $a\leq\frac{\ln(z_n)}{\ln(n+1)}$ and increases afterwards.
\end{proof}

\begin{lemma}\label{Lemma:Final}
$u_n(1)- u_n'(1)b_n \leq 0$.
\end{lemma}
\begin{proof} Substituting $u_n(1)=\ln(\frac{(n+2)n}{(n+1)^2})$, $u_n'(1)= \ln(\frac{n+2}{n+1})-\frac{1}{n}\sum_{i=1}^n\frac{\ln(i+1)}{i}$,
and $b_n=\frac{\ln(n+1)-\ln(n)}{\ln(n+1)+\sum_{i=1}^{n-1}\frac{\ln{i+1}}{i}}$,
the inequality $u_n(1)- u_n'(1)b_n \leq 0$ can be rewritten as $r_n\leq Q_n$ where
\begin{eqnarray*}
r_n&\triangleq&\mbox{$\frac{1}{n}\sum _{i=1}^n \frac{\ln (i+1)}{i}$}\\
Q_n&\triangleq&
\mbox{$ \frac{\left(1-\frac{1}{n}\right) \ln (n+1) \ln \left(\frac{(n+1)^2}{n (n+2)}\right)+\ln \left(\frac{n+2}{n+1}\right) \ln \left(\frac{n+1}{n}\right)}{\ln \left(\frac{n+1}{n}\right)-n \ln \left(\frac{(n+1)^2}{n (n+2)}\right)}$}.
\end{eqnarray*}
Since $r_1=\ln 2=Q_1$ we just need to show  that $r_n\leq Q_n$ for $n\geq 2$.
Now, using the inequality $\ln b\geq 1-\frac{1}{b}$ we get the lower bound 
$$\mbox{$
Q_n\geq q_n\triangleq \frac{\frac{n-1}{n} \ln (n+1)+\frac{n+1}{n+2}}{(n+1)^2 \ln \left(\frac{n+1}{n}\right)-n}
$}$$
so that it suffices to show that $r_n\leq q_n$. By direct evaluation one can check that $r_2\leq q_2$ 
so the conclusion follows from the fact that  $r_n$ decreases while $q_n$ increases.
Indeed, $r_n$ decreases since it is the average of the decreasing sequence $\frac{\ln(i+1)}{i}$,
whereas $q_n=\frac{f(n)}{g(n)}$ with 
$f(x)=\frac{x-1}{x} \ln (x+1)+\frac{x+1}{x+2}$ positive increasing and 
$g(x)=(x+1)^2 \ln \left(\frac{x+1}{x}\right)-x$ is positive and decreasing for $x\geq 2$. 
\end{proof}

\subsection{Bounds for  \texorpdfstring{$\frac{\delta_{n-1}}{1-\alpha_n}$}{Lg} in terms of  \texorpdfstring{$\tau_n$}{Lg}}
We proceed next to establish explicit bounds of the form $\delta_{n-1}\leq(1-\alpha_n)\varphi(\tau_n)$
required to apply  \cref{Thm:convolution_bound}. Since  \cref{Lemma:L3} gives
$\delta_{n-1}\leq\sqrt{\alpha_n}$, we will look for suitable functions $\varphi(\cdot)$ such that 
$$\frac{\delta_{n-1}}{1-\alpha_n}\leq \frac{\sqrt{\alpha_n}}{1-\alpha_n}\leq\varphi(\tau_n)\qquad(\forall\,n\geq 2).$$
We restrict to $a\in[\frac{1}{2},1]$ and treat separately the cases $\frac{1}{2}\leq a<1$ and $a=1$.
\subsubsection{Bounds for \texorpdfstring{$\frac{1}{2}\leq a<1$}{}}\label{appendix21}
Approximating the sum $\tau _n=\sum _{k=1}^n \alpha _k(1-\alpha_k)$ by an integral of $\frac{1}{(x+1)^{a}}(1-\frac{1}{(x+1)^{a}})$ one can 
check that $\tau_n\approx\frac{ 1}{1-a}(n+1)^{1-a}$ which yields $\alpha_n\approx \left((1-a)\tau_n\right)^{-\frac{a}{1-a}}$.
This prompts us to look for an estimate of the form
\begin{equation}\label{Eq:cotaalpha}
\mbox{$\frac{\delta_{n-1}}{1-\alpha_n}\leq \frac{\sqrt{\alpha_n}}{1-\alpha_n}\leq \varphi(\tau_n)=(\lambda_a+(1-a)\, \tau_n)^{-\frac{a}{2 (1-a)}}\qquad(\forall\,n\geq 2)$}
\end{equation}
with the largest possible constant $\lambda_a$, which requires 
$$\mbox{$\lambda_a\leq\xi_n\triangleq \left(\frac{1-\alpha_n}{\sqrt{\alpha_n}}\right)^{\!\!\frac{2 (1-a)}{a}}\!\!-(1\!-\!a)\, \tau_n\qquad(\forall\,n\geq 2).$}$$
\begin{lemma}\label{Lemma:kappa} The sequence $\xi_n$ is increasing and \eqref{Eq:cotaalpha} holds with optimal constant
$$\mbox{$\lambda_a=\xi_2=\left(3^{a/2}-3^{-a/2}\right)^{\frac{2 (1-a)}{a}} -(1\!-\!a) \,\tau_2(a)$.}$$
\end{lemma}
\begin{proof} We show that 
$\xi_n\geq\xi_{n-1}$ for $n\geq 3$. Denoting $f(x)=(x^{a/2}-x^{-a/2})^{\frac{2 (1-a)}{a}}$
we have
$$\mbox{$ \xi_n-\xi_{n-1}=f(n+1)-f(n)-(1-a)\alpha_n(1-\alpha_n).$}$$
A straightforward computation gives
\begin{eqnarray*}
f'(x)&=&\mbox{$\frac{(1-a) x^a (1-x^{-a})^{2/a} (x^a+1)}{(x^a-1)^3}$}\\
f''(x)&=&-\mbox{$\frac{(1-a) x^{a-1} (1-x^{-a})^{2/a} (ax^{2 a}+(4 a-2) x^a-(2-a))}{(x^a-1)^4}$}
\end{eqnarray*}
from where it follows that $f''(x)\leq 0$ for all $x\geq \left(\frac{1-2 a+\sqrt{1-2 a+3 a^2}}{a}\right)^{1/a}$\!\!.
This last expression is smaller than $3$ for all $a\in[\frac{1}{2},1]$, so that 
$f'(x)$ decreases for $x\geq 3$ and therefore
$$\mbox{$ f(n+1)-f(n)=\int_n^{n+1}f'(x)\,dx\geq f'(n+1).$}$$
The result will follow by proving that $f'(n+1)\geq(1-a)\alpha_n(1-\alpha_n)$, that is 
$$\mbox{$(1-a) \frac{ (n+1)^a (1-(n+1)^{-a})^{2/a} ((n+1)^a+1)}{((n+1)^a-1)^3}\geq (1-a)\frac{1}{(n+1)^a} (1-\frac{1}{(n+1)^a})$}.$$
Canceling the factor $(1-a)$ and setting $x=(n+1)^a$, this amounts to
$$\mbox{$\frac{ x (1-\frac{1}{x})^{2/a} (x+1)}{(x-1)^3}\geq \frac{1}{x} (1-\frac{1}{x})$}$$
which is in turn equivalent to
$$\mbox{$(\frac{x-1}{x})^{2/a}\geq \frac{(x-1)^4}{x^3(x+1)}$}.$$
The expression on the left increases with $a$ so it suffices to check at $a=\frac{1}{2}$, which
amounts to the trivial inequality $1\geq \frac{x}{x+1}$, completing the proof.
\end{proof}

\subsubsection{Bounds for \texorpdfstring{$a=1$}{1}}
For $a=1$, which corresponds to a simple averaging with $\alpha_n=\frac{1}{n+1}$, 
we have the following bound obtained as a limit of \eqref{Eq:cotaalpha}.
\begin{lemma}\label{Lemma:alfa1} Let $\omega=\frac{\sqrt{3}}{2}  \,e^{17/72}\approx 1.0966$. Then for $\alpha_n=\frac{1}{n+1}$ we have 
\begin{equation}\label{Eq:cotaalpha2}
\mbox{$\frac{\delta_{n-1}}{1-\alpha_n}\leq\frac{\sqrt{\alpha_n}}{1-\alpha_n}\leq\omega\, e^{-\tau_n/2}\qquad(\forall\,n\geq 2)$}
\end{equation}
\end{lemma}
\begin{proof} Since $\frac{\delta_{n-1}}{1-\alpha_n}$ and $\frac{\sqrt{\alpha_n}}{1-\alpha_n}$ are continuous in $a$,
\eqref{Eq:cotaalpha2} follows by letting $a\to 1$ in \eqref{Eq:cotaalpha}. To see this 
let us set $b_a\triangleq \frac{\lambda_a}{1-a}$ and
rewrite the right hand side of  \eqref{Eq:cotaalpha} as
$$ \mbox{$(\lambda_a+(1-a) \tau_n)^{-\frac{a}{2 (1-a)}}=\lambda_a^{-\frac{a}{2 (1-a)}}\left((1+ \frac{\tau_n}{b_a})^{-b_a}\right)^{\frac{a}{2\lambda_a}}$}.$$
For $a\to 1$ we have $\lambda_a\to 1$, hence $b_a\to\infty$ and $\frac{a}{2\lambda_a}\to \frac{1}{2}$, and therefore
\[
\left((1+ {\tau_n}/{b_a})^{-b_a}\right)^{\frac{a}{2\lambda_a}}\to e^{-\tau_n/2}.
\]
The proof is completed by a straightforward calculation that yields $\lambda_a^{-\frac{a}{2 (1-a)}}\to\omega$. 
\end{proof}

\vspace{2ex}
\begin{lemma}\label{Lemma:L41} 
Let $\tau_n(a)=\sum_{k=1}^n\frac{1}{(k+1)^a}(1-\frac{1}{(k+1)^a})$.
Then $\tau_n(a)\geq\tau_1(a)\;n^{1-a}$ for all $\frac{1}{2}\leq a<1$, while for $a=1$
we have $\tau_n(1)\geq\frac{1}{4\ln(2)}\ln(n+1)$.
\end{lemma}
\begin{proof}
For $n=1$ both inequalities are trivially satisfied with equality.
For the induction step, let us consider first the case $\frac{1}{2}\leq a<1$ and assume that the inequality holds for $n-1$. Then 
\begin{eqnarray*}
\tau_n(a)&=&(n+1)^{-a}\big(1-(n+1)^{-a}\big)+\tau_{n-1}(a)\\
&\geq& 
(n+1)^{-a}\big(1-(n+1)^{-a}\big)+\tau_1(a)\;(n-1)^a,
\end{eqnarray*}
hence it suffices to prove that 
$(n+1)^{-a}(1-(n+1)^{-a})\geq \tau_1(a)\,(n^a-(n-1)^a)
$,
or equivalently 
$$\frac{(1 - (n + 1)^{-a})}{(n + 1)^a (n^{1 - a} - (n - 1)^{1 - a})} \geq \tau_1(a).$$
Since the latter holds with equality for $n=1$, it suffices to show that the quotient
increases in $n$. 
The numerator is clearly increasing so it suffices to
show that the
denominator decreases. We claim that 
$(x + 1)^
   a (x^{1 - a} - (x - 1)^{1 - a})$ decreases for $x \geq 1$. Indeed,
differentiating with respect to $x$, this amounts to
$$a \left(1+\frac{x^a}{x^a-(x-1)^a}\right)\leq 1+x$$
and using the change of variable $z=(x-1)/x$ this is in turn equivalent to 
$$a\left(1+\frac{1}{1-z^a}\right)\leq 1+\frac{1}{1-z}\qquad\forall z\in[0,1).$$
The latter holds since $a\mapsto a\left(1+\frac{1}{1-z^a}\right)$ is increasing, which can be checked by differentiating with respect to $a$ and
using the inequality $(1-y)(2-y)+y \log (y)\geq 0$ for $y\in(0,1)$.

Similarly, for $a=1$, an inductive argument to prove 
$\tau_n(1)\geq\frac{1}{4\ln(2)}\ln(n+1)$ reduces the problem to showing that for all $n\geq 1$
$$\frac{1}{(n + 1)}\left (1 - \frac{1}{(n + 1)}\right) \geq \frac{1}{4\ln(2)} \left(\ln (n + 1) - \ln (n)\right).$$
This is in turn equivalent to 
$$ \frac{(n+1)^2}{n} \log \left(\frac{n+1}{n}\right)\leq 4\ln(2)$$
which follows directly since $x\mapsto \frac{(x+1)^2}{x} \log \left(\frac{x+1}{x}\right)$ is a decreasing function.
\end{proof}

\end{document}